\documentclass[12pt]{amsart}
\usepackage{hyperref}

\usepackage[utf8]{inputenc}

\usepackage{amssymb, graphicx}
\usepackage{amsfonts}
\usepackage{amsmath}
\usepackage{latexsym}
\usepackage{amscd}
\usepackage{verbatim}
\usepackage{mathrsfs}

\allowdisplaybreaks[4]

\newcommand{\mf}{\mathfrak}
\newcommand{\g}{\mf{g}}

\newcommand{\fa}{\mf{a}}

\newcommand{\gl}{\mf{gl}}

\newcommand{\p}{\mf{p}}

\allowdisplaybreaks[4]

\newcommand{\Z}{{\mathbb Z}}
\newcommand{\N}{{\mathbb N}}
\newcommand{\Q}{{\mathbb Q}}

\renewcommand{\phi}{\varphi}

\def\gl{\mathfrak{gl}}

\def\rad{\text{rad}}

 \def\ann{\rm ann}

\def\bC{\mathbb{C}}

\newcommand{\tW}{\tilde{W}}

\newtheorem{theorem}{Theorem}[section]

\newtheorem{lemma}[theorem]{Lemma}

\theoremstyle{remark}
\newtheorem{remark}[theorem]{Remark}

\numberwithin{equation}{section}
\def\Vir{\mathrm{Vir}}

\begin{document}




\title[Simple weight modules over  Witt superalgebras]{Simple weight modules with finite-dimensional weight spaces over  Witt superalgebras}
\author{Yaohui Xue, Rencai L\"{u}}
\maketitle







\begin{abstract}
Let $A_{m,n}$ be the tensor product of  the  Laurient polynomial algebra in  $m$ even variables and the exterior algebra in  $n$ odd variables over the complex field $\bC$, and the Witt superalgebra $W_{m,n}$ be the Lie superalgebra of superderivations of $A_{m,n}$. In this paper, we classify the simple weight $W_{m,n}$ modules with finite-dimensional weight spaces with respect to the standard Cartan algebra of $W_{m,0}$. Every such module is either a simple quotient of a tensor module or a module of highest weight type.\end{abstract}

\vskip 10pt \noindent {\em Keywords:}Witt superalgebra,  simple module, cuspidal module, weight module
\vskip 5pt
\noindent
{\em 2010  Math. Subj. Class.:} 17B10, 17B20, 17B65, 17B66, 17B68



\section{Introduction}

We denote by $\Z, \Z_+, \N, \Q$ and $\bC$ the sets of all integers, non-negative integers,  positive integers, rational numbers and complex numbers, respectively. All vector spaces and algebras in this paper are over $\bC$, and all modules over Lie superalgebras are $\Z_2$-graded. We denote by $U(\fa)$ the universal enveloping algebra of the Lie superalgebra $\fa$.

Let $A=A_{m,n}$ be the tensor superalgebra of  the  Laurient polynomial algebra in  $m$ even variables $t_1,t_2,\ldots,t_m$ and the exterior algebra in $n$ odd variables $\xi_1,\xi_2,\ldots, \xi_n$, and the Witt superalgebra $W=W_{m,n}$ be the Lie superalgebra of  superderivations of $A_{m,n}$. Denote by
$D_{m}={\rm span}\{ t_i\frac{\partial}{\partial t_i}|\,i=1,2,\ldots,m\}$ the Cartan subalgebra of $W_{m,0}$. Throughout this paper, a $W_{m,n}$ module $M$ is called a weight module if the action of $D_m$ on $M$ is diagonalizable.

Representation theory of Witt algebra $W_{m,0}$ has been well-developed.  Simple weight modules with finite-dimensional weight spaces (also called Harish-Chandra modules) for the Virasoro algebra (which is the universal central extension of $W_{1,0}$) were conjectured by V. Kac  in \cite{K2} and  classified by O.Mathieu in \cite{Ma}, see also \cite{Su2} for another approach.  Then similar classification was given for  the higher rank Virasoro algebras in \cite{Su1,LZ2}.  In 2004 Eswara Rao conjectured in \cite{E2} that a simple weight module for $W_{m.0}$ with finite-dimensional weight spaces is either a quotient of a tensor module or a module of highest weight type. The weight set of those modules for $W_{m,0}$ was given in \cite{MZ}. Finally, Y. Billig and V. Futorny completed the classification for $W_{m,0}$ in \cite{BF1}. The $A_{m,0}$-cover method developed in \cite{BF1} turns out to be extremely useful. For more related results, we refer the readers to \cite{BF2, E1, E2, Sh} and the references therein.

The finite-dimensional simple $W_{0,n}$ modules were classified in \cite{BL}. Very recently, the simple weight modules with finite-dimensional weight spaces over the $N=2$ Ramond algebra (which is a central extension of $W_{1,1}$)  were classified in \cite{Liu1}.

Hence it is  natural to consider the classification of simple weight modules over $W_{m,n}$ which have finite-dimensional weight spaces.

A  $W_{m,n}$ weight module is called {\it cuspidal} or {\it uniformly bounded} if  the dimensions of its weight spaces  are uniformly bounded by some constant.

This paper is arranged as follows. In Section 2, we collect some notations and results for later use. In Section 3, we classify simple cuspidal $W_{m,n}$ modules by extending the methods and results by  E. Rao, Y. Billig, V. Futorny for the Lie algebras $W_{m,0}$ to that for the Lie superalgebras $W_{m,n}$, see Theorem \ref{the3.11}. This is the main part of this paper, and the ideas in this section are used to solve the classification of simple cuspidal modules for the Lie algebra of vector fields on $\bC^n$, see \cite{LX}. In Section 4, we classify simple cuspidal modules over the extended Witt superalgebra by using the results in Section 3, see Theorem \ref{the4.4}. In Section 5, we classify simple weight $W_{m,n}$ modules with finite-dimensional weight spaces after  proving several auxiliary lemmas as those in \cite{MZ} for modules that are not cuspidal.  Every such module is either a quotient of a tensor module or a  module of highest weight type, see Theorem \ref{main}.

\section{Preliminaries}

In this section, we collect some basic definitions and results for our study.

 A vector superspace $V$ is a vector space endowed with a $\Z_2$-gradation $V=V_{\bar 0}\oplus V_{\bar 1}$. The parity of a homogeneous element $v\in V_{\bar{i}}$ is denoted by $|v|=\bar{i}\in \Z_2$. Throughout this paper, when we write $|v|$ for an element $v\in V$, we will always assume that $v$ is a homogeneous element.  Denote by $|I|$ the number of elements in the finite set $I$.

Denote by $d_i=t_i\frac{\partial}{\partial t_i},\forall i=1,2,\ldots,m$. Let $$\Delta=\Delta_{m,n}={\rm span}\{d_i,\frac{\partial}{\partial \xi_j}\ |\ i=1,\ldots,m;j=1,\ldots,n\}.$$


Let $e_1,\dots,e_m$ be the standard basis of $\Z^m$.

 For convenience, we will omit $\otimes$  in $A_{m,n}$, and write $t^{\alpha}:=t_1^{\alpha_1}t_2^{\alpha_2}\cdots t_m^{\alpha_m}$, $\xi_{i_1,i_2,\ldots,i_k}:=\xi_{i_1}\xi_{i_2}\cdots \xi_{i_k},\forall\, \alpha=(\alpha_1,\ldots,\alpha_m)\in \Z^m, i_1,i_2,\ldots,i_k\in \{1,2,\ldots,n\}$.

 For any subset $I=\{i_1,\ldots,i_k\}\subseteq \{1,2,\ldots,n\}$,  we denote $\underline{I}=(l_1,l_2,\ldots,l_k)$ if $\{i_1,i_2,\ldots,i_k\}=\{l_1,l_2,\ldots,l_k\}$ and $l_1<\dots<l_k$. Denote $\xi_{I}:=\xi_{l_1,\ldots,l_k}$. We set $\xi_{\varnothing}=1$.

 Then $W_{m,n}=A_{m,n}\Delta$ has a standard basis $$\{t^{\alpha}\xi_Id_i, t^{\alpha}\xi_I\frac{\partial}{\partial \xi_j}\ |\ i=1,2,\ldots,m;j=1,2,\ldots,n; \alpha\in\Z^m; I\subseteq \{1,2,\ldots,n\}\}.$$

  We will also need the extended Witt superalgebra $\tW_{m,n}=W_{m,n}\ltimes A_{m,n}$ with the brackets
$$[a,a']=0,\ [x,a]=-(-1)^{|x||a|}[a,x]=x(a),\ \forall a,a'\in A_{m,n},\ x\in W_{m,n}.$$

Let $\g$ be any of $\tW_{m,n}$, $W_{m,n}$ or any Lie supersubalgebra of $\tW_{m,n}$ that contains $D_m$. A $\g$ module $M$ is called a {\it weight} module provided that the action of $D_m$ on $M$ is diagonalizable. Let $M$ be a weight $\g$ module. Then
\begin{equation}M=\oplus_{\lambda\in \bC^m} M_{\lambda},\end{equation} where $M_{\lambda}=\{v\in M| d_i v=\lambda_i v,i=1,2,\ldots,m\}$.  $M_{\lambda}$ is called a weight space corresponding to the weight $\lambda$. The support ${\rm Supp}(M)$ of $M$ is defined as the set of all weights $\lambda$ for which $M_{\lambda}\ne 0$. In particular $\tW_{m,n}$ as a weight module over itself has support $\Z^m$. Therefore,  for any indecomposable weight $\g$ module we have  ${\rm Supp}(M)\subseteq \lambda+\Z^m$ for some $\lambda\in\bC^m$. A weight $\g$ module $M$ is called  {\it cuspidal} or {\it uniformly bounded} provided that there exists a positive integer $N$ such that $\dim V_{\lambda}<N$ for all $\lambda\in {\rm Supp}(V)$. For any subset $S$ of $\bC^m$,  denote $M_{S}=\oplus_{\lambda\in S} M_S$. In particular, $\g_G$ is a Lie supersubalgebra of $\g$ if $G$ is a semi-subgroup of the addtive group $\bC^m$.

For any $\mu\in \bC^m$, denote $d_{\mu}=\mu_1d_1+\cdots+\mu_md_m$. $\mu$ is called {\it generic} if $\mu_1,\mu_2,\ldots,\mu_m$ are  linearly independent over $\Q$. For a given generic $\mu$, we denote by $\Vir[\mu]=A_{m,0}(\mu_1d_1+\ldots+\mu_m d_m)$ the solenoidal Lie algebra (also called the centerless higher rank Virasoro algebra). A $\Vir[\mu]$ module $M$ is called a weight module if  the action of $d_{\mu}$ on $M$ is diagonalizable.

Let $\sigma:L\rightarrow L'$ be any homomorphism of Lie superalgebras or associative superalgebras, and $M$ be any $L'$ module. We make $M$ into an $L$ module by $x\cdot v=\sigma(x) v,\forall x\in L, v\in M$. The resulting module will be denoted by $M^{\sigma}$.  Denote by $T$ the automorphism of $L$ defined by $T(x)=(-1)^{|x|}x,\forall x\in L$.  For any $L$ module $M$, we can make it into a new module  $\Pi(M)$ by a parity-change of $M$.

Let $B$ be any associative superalgebra. A  $B$ module $M$ is called {\it strictly simple} if  it is a simple module over the associative algebra $B$ (forgetting the $\Z_2$-gradation), i.e., $M$ has no $B$ invariant  subspaces (not necessarily $\Z_2$-graded) except $0$ and $M$.

We will need the following two results on tensor modules over tensor superalgebras.

\begin{lemma}\label{lem2.1}Let $B, B'$ be unital associative superalgebras, and $M,M'$ be $B,B'$ modules, respectively. Then
$M\otimes {M'}\cong \Pi(M)\otimes \Pi({M'}^{T})$  as $B\otimes B'$ modules.

\end{lemma}

\begin{proof}  It follows directly from $(b\otimes b')(v\otimes v')=(-1)^{|v||b'|} bv\otimes b'v'\in M\otimes M'$ and $(b\otimes b')(v\otimes v')=(-1)^{|b'|(|v|-1)}(bv)\otimes ((-1)^{|b'|}b'v')=(-1)^{|v||b'|} bv\otimes {b'v'}\in \Pi(M)\otimes \Pi({M'}^{ T})$.  \end{proof}

\begin{lemma}\label{lem2.2}Let $B, B'$ be unital associative superalgebras such that $B'$ has a countable basis, $R=B\otimes B'$,
$M'$ be a strictly simple $B'$ module and $M$ be  a $B$ module. Then
 \begin{itemize}
\item[(1).] Any $R$ submodule of  $M\otimes M'$ is of the form $N\otimes M'$ for some $B$ submodule $N$ of $M$;

\item[(2).] Any simple quotient of the $R$ module $M\otimes M'$ is isomorphic to some $K\otimes M'$ for some simple quotient $K$ of $M$;

 \item[(3).] $M\otimes M'$ is a simple $R$ module if and only if $M$ is a simple $B$ module;

\item[(4).] Suppose that $V$ is a simple $R$ module and $V$ contains a  $B'=\bC\otimes B'$  submodule $M'$ that is strictly simple . Then $V\cong M\otimes M'$ for some simple $B$ module $M$.
\end{itemize}
\end{lemma}

\begin{proof}  Let $F$ be any $R$ submodule of  $M\otimes M'$. Then for any nonzero homogeneous  vector $v\in F$, we may write $v=\sum_{i=1}^k {w_i}\otimes w_i'$ with $w_1,\ldots, w_k$ and $w_1',\ldots,w_k'$ being homogeneous elements respectively, and $w_1',\ldots,w_k'$ being linearly independent.  Since $B'$ has a countable basis and $M'$ is a strictly simple $B'$ modules, from the Theorem of Density, there exists a $b'\in B'$ such that $b'w_i'=\delta_{i,1}w_i', i=1,2,\ldots,k.$ Moreover, since $w_i'$ are homogeneous, we may assume that $b'$ is homogeneous. Therefore, $b'v=b'(\sum_{i=1}^k w_i\otimes w_i')=(-1)^{|w_1||b'|} w_1\otimes w_1' \in F$. So $(\bC\otimes B')(w_1\otimes w_1')=\bC w_1\otimes B'w_1'=\bC w_1\otimes M'\subseteq F$. Similarly, we have $\bC w_i\otimes M'\subseteq F$. Now we have proved  $F=K\otimes M'$, where $K=\{w\in M\ |\ w\otimes M'\subseteq F\}$. It is clear that $K$ is a $B$ submodule of $M$ and we have (1).

(2) and (3) follow easily from (1).  Suppose that $V$ is a simple $R$ module and $V$ contains a  $\bC\otimes B'$  submodule $M'$ that is strictly simple. Then $V$ is a simple quotient of ${\rm Ind}_{B'}^{B\otimes B'} M'\cong B\otimes M'$. Now (4) follows from (2).
\end{proof}
\begin{remark} We do not have the results in the lemma if $M'$ is a simple instead of strictly simple $B'$ module. See for example  \cite[Section 3.1]{CW}. \end{remark}

Let us briefly introduce exp-polynomial Lie superalgebras and exp-polynomial modules as in \cite{BZ}.

A Lie superalgebra $L$  is called $\Z^m$-graded if $L=\oplus_{\alpha\in \Z^m} L_{\alpha}$ as superspaces and $[L_{\alpha,\bar{i}},L_{\alpha',\bar{i'}}]\subseteq L_{\alpha+\alpha',\bar{i}+\bar{i}'},\forall \alpha,\alpha'\in \Z^m,\bar{i},\bar{i}'\in \Z_2$.  Let $K=K_{\bar{0}}\sqcup K_{\bar{1}}$ be an index set. Then $L$ is said to be a $\Z^m$-graded {\it exp-polynomial Lie superalgebra} if $L$ has a spanning set $\{g_{k}(\alpha)|k\in K, \alpha\in \Z^m\}$ with $g_{k}(\alpha)\in  L_{\alpha,\bar{i}},\forall k\in K_{\bar{i}} $, and there exists a family of exp-polynomial  functions $\{f_{k,k'}^s(\alpha,\alpha')|k,k',s\in K\}$ in $2m$ variables $\alpha_i,\alpha_i'$ and where for each $k,k'$ the set $\{s|f_{k,k'}^s(\alpha,\alpha')\ne 0\}$ is  finite, such that $[g_{k}(\alpha),g_{k'}(\alpha')]=\sum_{s\in K}f_{k,,k'}^s(\alpha,\alpha')g_{s}(\alpha+\alpha'), \forall k,k'\in K,\alpha,\alpha'\in \Z^m.$ An $L$ module $V$ is call $\Z^m$-graded  if $V=\oplus_{\alpha\in \Z^m} V_{\alpha}$ as superspaces and $L_{\alpha,\bar{i}}V_{\beta,\bar{j}}\subseteq V_{\alpha+\beta,\bar{i+j}},\forall \bar{i,}\bar{j}\in \Z_2, \alpha,\beta\in \Z^m$. Let $J=J_{\bar{0}}\sqcup J_{\bar{1}}$ be a finite index set. Then a $\Z^m$-graded $L$ module $V$ is called {\it $\Z^m$-graded exp-polynomial modules} if  $V$ has a spanning set $\{v_j(\alpha)|j\in J, \alpha \in \Z^m\}$ with $v_j(\alpha)\in V_{\alpha,\bar{i}},\forall j\in J_{\bar{i}}$, and there exists a family of exp-polynomial functions $h_{k,j}^{j'}(\alpha,\beta)$ for $k\in K, j,j'\in J$ such that $g_{k}(\alpha)v_{j}(\beta)=\sum_{j'\in J}h_{k,j}^{j'}(\alpha,\beta)v_{j'}(\alpha+\beta)$, where for each $k,j$ the set $\{j'|h_{k,j}^{j'}(\alpha,\beta)\ne 0\}$ is finite. A $\Z^m$-graded exp-polynomial Lie superalgebra $L$ is called {\it $\Z^m$-extragraded} if $L$ has another $\Z$-gradation \begin{equation}\label{extragrading}L=\oplus_{s\in\Z} L^{(s)}\end{equation} and the set $K$ is a disjoint union of finite subsets $K_s$ such that $\{g_{k}(\alpha)|k\in K_{s}, \alpha\in \Z^m\}$ spans the vector superspace $L^{(s)}$ for each $s\in \Z$.

Assume that $V=\oplus_{\alpha\in \Z^m} V_{\alpha}$ is a $\Z^m$-graded  exp-polynomial $L^{(0)}$ module. we can define the action of $L^+=\oplus_{i\in \N} L^{(i)}$ on $V$ by $L^+V=0$ and then consider the induced module $\tilde{M}(V)={\rm Ind}_{L^{(0)}+L^+}^L V\cong U(L^{-})\otimes V$. It is clear that $\tilde{M}(V)$ is a $\Z^{m+1}$-graded module over $L$. And $\tilde{M}(V)$ has a unique maximal proper $\Z^{m+1}$-graded submodule $\tilde{M}^{rad}$ which intersects trivialy with $V$. Let $M(V)=\tilde{M}(V)/\tilde{M}^{\rad}$.

\begin{lemma}\label{exp}\cite[Theorem 1.5]{BZ} Let $L$ be a $\Z^m$-extragraded Lie superalgebra with grading (\ref{extragrading}),  and $V$ be a $\Z^m$-graded exp-polynomial $L^{(0)}$ module. Then the $\Z^{m+1}$-graded $L$ module $M(V)$ has finite-dimensional $\Z^{m+1}$-graded spaces.\end{lemma}
\begin{proof}The proof is similar to that of Theorem 1.5 in \cite{BZ}. \end{proof}

\section{Cuspidal modules}



In this section, we will classify simple cuspidal $W_{m,n}$ modules. Let us fix a $ (m,n)\in  \Z_+^2\backslash \{(0,0)\}$. Denote $A=A_{m,n}, W=W_{m,n}, \tW=\tW_{m,n}$, and $\Delta=\Delta_{m,n}$ for short.

A $\tW$ module $M$ is called an $AW$ module if the action of $A$ on $M$ is associative, i.e.,
$$a\cdot a'\cdot v=(aa')\cdot v, t^0\cdot  v=v,\, \forall a,a'\in A,\ v\in M.$$

 Denote by  $\tau(i_1,\ldots,i_k)$  the inverse order of the sequence $i_1,\ldots,i_k$, and $\tau(I,J):=\tau(\underline{I},\underline{J})=\tau(k_1,\dots,k_s,l_1,\dots,l_r)$ when $I\cap J=\varnothing$, where $\underline{I}=(k_1,\dots,k_s),\underline{J}=(l_1,\dots,l_r)$. We set $\tau(\varnothing,\varnothing)=\tau(\varnothing)=0$. Denote $\xi_{I,J}=\xi_I\xi_J$. Then $\xi_{I\cup J}=(-1)^{\tau(I,J)}\xi_I\xi_J$ for all $I\cap J=\varnothing$.

 Let  $\mathcal{J}$ be the left ideal of $U(\tW)$ generated by $\{t^0-1,t^{\alpha}\xi_I\cdot t^{\beta}\xi_J-t^{\alpha+\beta}\xi_I\xi_J|\alpha,\beta\in \Z^m, I,J\subseteq \{1,2,\ldots,n\} \}$. Then it is easy to see that $\mathcal{J}$ is in fact an ideal of $U(\tW)$. Now we have the quotient algebra $\bar{U}_{m,n}=U(\tW)/{\mathcal{J}}=(U(A) U(W))/\mathcal{J}$. From PBW Theorem,  we may identify $A$, $W$ with their images in $\bar{U}=\bar{U}_{m,n}$. Thus $\bar{U}=A\cdot U(W)$.   And denote by $K_{m,n}$ the associative supersubalgebra of $\bar{U}$ generated by $A$ and $\Delta$, which is the Weyl superalgebra, see \cite{SZZ}.

$A\cdot W$ is a Lie supersubalgebra of $\bar{U}$ with a basis $${\{}t^{\alpha}\xi_I\cdot t^{\beta}\xi_Jd_i,t^{\alpha}\xi_I\cdot t^{\beta}\xi_J\frac{\partial}{\partial \xi_j}| i=1,2,\ldots,m;j=1,2,\ldots,n;\alpha,\beta\in \Z^m;I,J\subseteq \{1,2,\ldots,n\}{\}}$$ and the brackets
\begin{equation*}[f\cdot b, g\cdot d]=f[b,g]\cdot d-(-1)^{|f\cdot b||g\cdot d|}g[d,f]\cdot b+(-1)^{|b||g|}fg\cdot [b,d],\end{equation*} $\forall f,g\in A,b,d\in W$.

 Let
\begin{eqnarray*}
&X_{\alpha,i}=t^{-\alpha}\cdot t^{\alpha}d_i-d_i,Y_{\alpha,j}=t^{-\alpha}\cdot t^{\alpha}\frac{\partial}{\partial\xi_j}-\frac{\partial}{\partial\xi_j},\\
&X'_{\alpha,I,i}=\sum_{J\subseteq I}(-1)^{\tau(J,I\setminus J)+|J|}t^{-\alpha}\xi_{J}\cdot t^{\alpha}\xi_{I\setminus J} d_i,\\
&Y'_{\alpha,I,j}=\sum_{J\subseteq I}(-1)^{\tau(J,I\setminus J)+|J|}t^{-\alpha}\xi_{J}\cdot t^{\alpha}\xi_{I\setminus J} \frac{\partial}{\partial\xi_j},\\ & \forall  \alpha\in\Z^m,\ i\in\{1,\dots,\ m\},\ j\in\{1,\dots,\ n\},  I\subseteq \{1,2,\ldots,n\}.
\end{eqnarray*}

Let $\mathcal{T}=\mathcal{T}_{m,n}$ be the supersubspace of $\bar{U}$ with a basis $\mathcal{B}$ consisting of $X_{\alpha,i},Y_{\alpha,j},X'_{\beta,I,i},Y'_{\beta,I,j}$, where $\alpha\in \Z^m\setminus\{0\}, \beta\in\Z^m, i\in\{1,\dots,\ m\},\ j\in\{1,\dots,\ n\}, \varnothing\ne I\subseteq \{1,2,\ldots,n\}$.
\begin{lemma} \label{lem3.1}Let $\underline{I}=(l_1,\ldots,l_k)$ and $s\in \{1,\ldots,k\}$. We have
\begin{itemize}
\item[(1).] $t^{-\beta}\cdot t^{\beta}\xi_I d_i=\sum_{J\subseteq I}(-1)^{\tau(J,I\setminus J) }\xi_{J}\cdot  X'_{\beta,I\setminus J,i}$;

\item[(2).]$t^{-\beta}\cdot t^{\beta}\xi_I\frac{\partial}{\partial\xi_j}=\sum_{J\subseteq I}(-1)^{\tau(J,I\setminus J)} \xi_{J}\cdot  Y'_{\beta,I
\setminus J,j}$;

\item[(3).]$\sum_{J\subseteq I}(-1)^{\tau(J,I\setminus J)+|J|}t^{-\beta}\frac{\partial \xi_{J}}{\partial \xi_{l_s}}\cdot t^{\beta}\xi_{I\setminus J} d_i=(-1)^{s}X'_{\beta,I\setminus\{l_s\},i}$;

\item[(4).]$\sum_{J\subseteq I}(-1)^{\tau(J,I\setminus J)+|J|}t^{-\beta}\frac{\partial \xi_{J}}{\partial \xi_{l_s}}\cdot t^{\beta}\xi_{I\setminus J}\frac{\partial}{\partial \xi_{j}}=(-1)^{s}Y'_{\beta,I\setminus\{l_s\},j}$;

\item[(5).]$\sum_{J\subseteq I}(-1)^{\tau(J,I\setminus J)}t^{-\beta}\xi_{J}\cdot t^{\beta}\frac{\partial \xi_{I\setminus J}}{\partial \xi_{l_s}}d_i=(-1)^{s-1} X'_{\beta,I\setminus\{l_s\},i}$;

\item[(6).]$\sum_{J\subseteq I}(-1)^{\tau(J,I\setminus J)}t^{-\beta}\xi_{J}\cdot t^{\beta}\frac{\partial \xi_{I\setminus J}}{\partial \xi_{l_s}} \frac{\partial}{\partial \xi_{j}}=(-1)^{s-1} Y'_{\beta,I\setminus\{l_s\},j}$.
\end{itemize}
\end{lemma}
\begin{proof}  First of all, from

\begin{align*}&\Big(\sum_{J\subseteq K} (-1)^{\tau(J,I\setminus J)+\tau(K\setminus J, I\setminus K)+|K\setminus J|+\tau(J,K\setminus J)}\Big)\xi_{K}\xi_{I\setminus K}\\
&=\sum_{J\subseteq K} (-1)^{\tau(J,I\setminus J)+\tau(K\setminus J, I\setminus K)+|K\setminus J|}\xi_{J} \xi_{K\setminus J}\xi_{I\setminus K}\\ &=\sum_{J\subseteq K} (-1)^{\tau(J,I\setminus J)+|K\setminus J|}\xi_{J} \xi_{(K\setminus J)\cup (I\setminus K)}\\& =\sum_{J\subseteq K} (-1)^{\tau(J,I\setminus J)+|K\setminus J|}\xi_{J} \xi_{I\setminus J}\\ &=\sum_{J\subseteq K} (-1)^{|K\setminus J|}\xi_{I}=0,\forall K\ne \varnothing,\end{align*} we have $\sum_{J\subseteq K} (-1)^{\tau(J,I\setminus J)+\tau(K\setminus J, I\setminus K)+|K\setminus J|+\tau(J,K\setminus J)}=0,\forall \varnothing \ne K\subseteq I$.
Then \begin{align*}&\sum_{J\subseteq I}(-1)^{\tau(J,I\setminus J) }\xi_{J}\cdot  X'_{\beta,I\setminus J,i}\\
&= \sum_{J\subseteq I}\sum_{J'\subseteq I\setminus J}(-1)^{\tau(J,I\setminus J) }(-1)^{\tau(J', I\setminus (J\cup J'))+|J'|}\xi_{J}\cdot  t^{-\beta}\xi_{J'}\cdot t^{\beta}\xi_{I\setminus (J\cup J')}d_i\\
&=\sum_{K\subseteq I}\sum_{J\subseteq K}(-1)^{\tau(J,I\setminus J)+\tau(K\setminus J, I\setminus K)+|K\setminus J|}\xi_{J}\cdot  t^{-\beta}\xi_{K\setminus J}\cdot t^{\beta}\xi_{I\setminus K}d_i\\
&=\sum_{K\subseteq I}\sum_{J\subseteq K}(-1)^{\tau(J,I\setminus J)+\tau(K\setminus J, I\setminus K)+|K\setminus J|+\tau(J,K\setminus J)} t^{-\beta}\xi_{K}\cdot t^{\beta}\xi_{I\setminus K}d_i\\
&=t^{-\beta}\cdot t^{\beta}\xi_I d_i+\sum_{\varnothing\ne K\subseteq I}\Big(\sum_{J\subseteq K}(-1)^{\tau(J,I\setminus J)+\tau(K\setminus J, I\setminus K)+|K\setminus J|+\tau(J,K\setminus J)}\Big) t^{-\beta}\xi_{K}\cdot t^{\beta}\xi_{I\setminus K}d_i\\
&=t^{-\beta}\cdot t^{\beta}\xi_I d_i.\end{align*} We get (1).

Similarly we have (2).

For any $\underline{J}=(l_{i_1},\ldots l_{i_p})$ with $i_q=s$ for some $1\le q\le p$,  there are exactly $(s-1)-(q-1)=s-q$ elements in $I\setminus J$ smaller than $l_{s}$. So  $\tau(J,I\setminus J)=s-q+\tau(J\setminus \{l_s\}, I\setminus J)$. Then  $(-1)^{\tau(J,I\setminus J)+|J|}\frac{\partial \xi_{J}}{\partial \xi_{l_s}}=(-1)^{\tau(J\setminus\{l_s\},I\setminus J)+s-q+|J\setminus\{l_s\}|+1}(-1)^{q-1} \xi_{J\setminus\{l_s\}}=(-1)^{\tau(J\setminus\{l_s\},I\setminus J)+|J\setminus\{l_s\}|+s} \xi_{J\setminus\{l_s\}}$. Thus
 \begin{align*}&\sum_{J\subseteq I}(-1)^{\tau(J,I\setminus J)+|J|}t^{-\beta}\frac{\partial \xi_{J}}{\partial \xi_{l_s}}\cdot t^{\beta}\xi_{I\setminus J} d_i\\
 &=\sum_{l_s\in J\subseteq I}(-1)^{\tau(J\setminus\{l_s\},I\setminus J)+|J\setminus\{l_s\}|+s}t^{-\beta}\xi_{J\setminus \{l_s\}}\cdot t^{\beta}\xi_{I\setminus J} d_i=(-1)^{s}X'_{\beta,I\setminus\{l_s\},i}.
\end{align*} So we get (3). Similarly we have (4). Now for any $\underline{I\setminus J}=(l_{i_{p+1}},\ldots,l_{i_k})$ with $i_{q}=s$ for some $p+1\le q\le k$, there are exactly $(k-s)-(k-q)$ elements in $J$ greater than $l_s$. Thus $\tau(J,I\setminus J)=q-s+\tau(J, I\setminus (J\cup \{l_s\}))$ and $(-1)^{\tau(J,I\setminus J)}\frac{\partial \xi_{I\setminus J}}{\partial \xi_{l_s}} =(-1)^{\tau(J, I\setminus (J\cup \{l_s\}))+q-s}\frac{\partial \xi_{I\setminus J}}{\partial \xi_{l_s}}=(-1)^{\tau(J, I\setminus (J\cup \{l_s\}))+q-s}(-1)^{q-|J|-1}\xi_{ I\setminus (J\cup \{l_s\})}=(-1)^{\tau(J, I\setminus (J\cup \{l_s\}))+|J|+s-1}\xi_{ I\setminus (J\cup \{l_s\})}$, from which we may easily deduce (5) and (6).
\end{proof}

\begin{lemma}\label{lem3.2}{\rm (1).} $\mathcal{T}=\{x\in A\cdot W|[x, \Delta]=[x,A]=0\}$. Thus  $\mathcal{T}$ is a Lie supersubalgebra of $\bar U$.

{\rm (2).} $\tilde{\mathcal{B}}=\mathcal{B}\cup \{d_i,\frac{\partial}{\partial \xi_j}|i=1,\ldots,m;j=1,\ldots,n\}$ is a basis of the free left $A$ module $A\cdot W$.\end{lemma}
\begin{proof}From Lemma \ref{lem3.1} (1) and (2), we know that $\mathcal{B}\cup \{d_i,\frac{\partial}{\partial \xi_j}|i=1,\ldots,m;j=1,\ldots,n\}$ is a generating set of the free left $A$ module $A\cdot W$. And it is straightforward to verify that $\mathcal{B}\cup \{d_i,\frac{\partial}{\partial \xi_j}|i=1,\ldots,m;j=1,\ldots,n\}$ is $A$-linearly independent. So we have (2).  Denote $T_1=\{x\in A\cdot W|[x, \Delta]=[x,A]=0\}$. It is easy to see that $[\mathcal{T},A]=0$, and from Lemma \ref{lem3.1} (3)-(6), we have $[\mathcal{T},\Delta]=0$. So $\mathcal{T}\subseteq T_1$.
Using (2), for any $x\in T_1$, write $x=\sum_{i=1}^k f_i\cdot x_i+x'$ with $f_i\in A, x_i\in \mathcal{B},x'\in A\cdot \Delta$. Then $[d,x]=\sum_{i=1}^k[d,f_i]\cdot x_i+[d,x']=0,\forall d\in \Delta$. That is $[\Delta, f_i]=[\Delta,x']=0,\forall i=1,2,\ldots,k$. So we have $f_i\in \bC$ and $x'\in \Delta$. And from $[x, f]=[x',f]=0,\forall f\in A$, we have $x'=0$. Now $T_1\subseteq \mathcal{T}$. So we have proved $\mathcal{T}=T_1$, which is (1).
\end{proof}

\begin{lemma}\label{lem3.3} We have the   associative superalgebra isomorphism \begin{equation}\label{iota}
\iota:K_{m,n}\otimes U(\mathcal{T})\rightarrow \overline{U}, \  \
\iota(x \otimes y)=x\cdot y,
\end{equation}
where $x\in K_{m,n}, y\in U(\mathcal{T})$.
\end{lemma}

\begin{proof} Note that $\mathcal{T}$ is a Lie supersubalgebra of $\bar{U}$ and $K_{m,n}$ is an associative supersubalgebra of $\bar{U}$. So the restrictions of $\iota$ on $K_{m,n}$ and $U(\mathcal{T})$ are well-defined. From Lemma \ref{lem3.2}, $\iota(K_{m,n})$ and $\iota(U(\mathcal{T}))$  are super commutative in $\bar{U}$. Hence $\iota$ is a well-defined  homomorphism of associative superalgebras.  Let  $W'=A\otimes \mathcal{T}+(A\cdot \Delta+A)\otimes \bC \subseteq K_{m,n}\otimes U(\mathcal{T})$. From Lemma \ref{lem3.1} (1) and (2), it is straightforward to verify that $\iota'=\iota|_{W'}:W'\rightarrow  A\cdot W+A$ is bijective hence a Lie superalgebra isomorphism. Therefore, the restriction of ${\iota'}^{-1}$ to $\tW=W+A$ gives a Lie superalgebra homomorphism $\eta: \tW\rightarrow K_{m,n}\otimes U(\mathcal{T})$ with
\begin{equation}\begin{split}&\eta(t^{\beta}\xi_I)=t^{\beta}\xi_I\otimes 1;\\
&\eta(t^{\beta}\xi_Id_i)=\sum_{ J\subsetneq I}(-1)^{\tau(J,I\setminus J) }t^{\beta}\xi_{J}\otimes  X'_{\beta,I\setminus J,i}+t^{\beta}\xi_I\otimes  X_{\beta,i}+(t^{\beta}\xi_I\cdot d_i)\otimes 1;\\
&\eta(t^{\beta}\xi_I\frac{\partial}{\partial\xi_j})=\sum_{J\subsetneq I}(-1)^{\tau(J,I\setminus J) }t^{\beta}\xi_{J}\otimes  Y'_{\beta,I\setminus J,j}+t^{\beta}\xi_I\otimes  Y_{\beta,j}+(t^{\beta}\xi_I\cdot\frac{\partial}{\partial\xi_j})\otimes 1.\end{split}
\end{equation}
So we have the associative superalgebra homomorphism $\tilde{\eta}:U(\tW)\rightarrow  K_{m,n}\otimes U(\mathcal{T})$ with $\tilde{\eta}|_{\tW}=\eta$. And it is clear that $\mathcal{J}\subseteq {\rm Ker}(\tilde{\eta})$. Hence we have the induced associative superalgebra homomorphism $\bar{\eta}:\bar{U}\rightarrow K_{m,n}\otimes U(\mathcal{T})$. It is clear that $\bar{\eta}=\iota^{-1}$ and $\iota$ is an isomorphism.
 \end{proof}

Let $\mathfrak{m}=\mathfrak{m}_{m,n}$ be the maximal ideal of $A_{m,n}$ generated by $t_i-1, \xi_j, i=1,2,\ldots,m;j=1,2,\ldots,n$.  Then $\mathfrak{m}\Delta$ is a Lie super subalgebra of $W=A\Delta$. And $\mathfrak{m}\Delta$ has a basis consisting of
\begin{equation}
(t^\alpha-1)d_i,(t^\alpha-1)\frac{\partial}{\partial\xi_j},t^\beta\xi_{I}d_i,
t^\beta\xi_{I}\frac{\partial}{\partial\xi_j},\end{equation}
$\alpha\in\Z^m\setminus \{0\}, \beta\in \Z^m,\, i=1,\dots,m,\ j=1,\dots,n; \varnothing\ne I \subseteq\{1,\dots,n\}.$

In general, the Lie brackets in $\mathcal{T}$ is hard to compute out directly even for $W_{m,0}$, see \cite{E2}. Here we construct a new isomorphism from $\mathcal{T}$ to $\mathfrak{m}\Delta$, which we believe is useful for further study on various Lie superalgebras of Cartan type.

Define a linear map $\psi:\mathcal{T}\rightarrow  \mathfrak{m}\Delta$ by
\begin{equation}\begin{split}&\psi(X_{\alpha,i})=(t^\alpha-1)d_i, \psi(X'_{\alpha,I,i})=t^\alpha\xi_{I}d_i,\\
 &\psi(Y_{\alpha,i})=(t^\alpha-1)\frac{\partial}{\partial\xi_i},\psi(Y'_{\alpha,I,i})=t^\alpha\xi_{I}\frac{\partial}{\partial\xi_i}. \end{split}\end{equation}
$\psi$ is clearly  an isomorphism of vector superspaces. In fact, we have

\begin{theorem}\label{the3.4}$\psi:\mathcal{T}
\rightarrow \mathfrak{m}\Delta$ is an isomorphism of Lie superalgebras.\end{theorem}

\begin{proof}From Lemma \ref{lem3.1} (1)(2), we have
\begin{equation}\label{eq3.5}\begin{split}
&(t^{\alpha}-1)d_i=t^{\alpha}\cdot X_{\alpha,i}+(t^{\alpha}-1)\cdot d_i;\\&(t^{\alpha}-1)\frac{\partial}{\partial\xi_j}=t^{\alpha}\cdot Y_{\alpha,i}+(t^{\alpha}-1)\cdot\frac{\partial}{\partial\xi_j};\\&t^{\alpha}\xi_Id_i=\sum_{ J\subsetneq I}(-1)^{\tau(J,I\setminus J) }t^{\alpha}\xi_{J}\cdot  X'_{\alpha,I\setminus J,i}+t^{\alpha}\xi_I\cdot  X_{\alpha,i}+t^{\alpha}\xi_I\cdot d_i ;\\&t^{\beta}\xi_{I'}\frac{\partial}{\partial\xi_j}=\sum_{J'\subsetneq I'}(-1)^{\tau(J',I'\setminus J') }t^{\beta}\xi_{J'}\cdot Y'_{\beta,I'\setminus J',j}+t^{\beta}\xi_{I'}\cdot Y_{\beta,j}+t^{\beta}\xi_{I'}\cdot\frac{\partial}{\partial\xi_j}.\end{split}
\end{equation}

Hence we have $\mathfrak{m}\Delta\subseteq \mathfrak{m}\cdot \Delta+A\cdot \mathcal{T}$. However $\mathfrak{m}\cdot \Delta+\mathfrak{m}\cdot \mathcal{T}$ is clearly an ideal of $\mathfrak{m}\cdot \Delta+A\cdot \mathcal{T}$, so we have the Lie superalgebra homomorphism $\omega:\mathfrak{m}\Delta\subseteq \mathfrak{m}\cdot \Delta+A\cdot \mathcal{T}\rightarrow (\mathfrak{m}\cdot \Delta+A\cdot \mathcal{T})/ (\mathfrak{m}\cdot \Delta+\mathfrak{m}\cdot \mathcal{T}) \rightarrow  (A\cdot \mathcal{T})/(\mathfrak{m} \cdot \mathcal{T}) \rightarrow \mathcal{T}$. More precisely,  from (\ref{eq3.5}), we have
\begin{equation}\begin{split}&\omega((t^\alpha-1)d_i)=X_{\alpha,i}, \omega(t^\alpha\xi_{I}d_i))=X'_{\alpha,I,i},\\
 &\omega((t^\alpha-1)\frac{\partial}{\partial\xi_i})=Y_{\alpha,i},\omega(t^\alpha\xi_{I}\frac{\partial}{\partial\xi_i})=Y'_{\alpha,I,i}. \end{split}\end{equation}
Hence $\psi=\omega^{-1}$ and $\omega$ is a Lie superalgebra isomorphism, so is $\psi$.
\end{proof}

For any $\lambda\in \bC^m$, let $\sigma_{\lambda}$ be the automorphism of the associative superalgebra $K_{m,n}$ with $\sigma_{\lambda}(d_i)=d_i+\lambda_i, \sigma_{\lambda}(\frac{\partial}{\partial \xi_j})=\frac{\partial}{\partial \xi_j}, \sigma_{\lambda}|_{A}={\rm id}_{A}$. Denote $A(\lambda):=A^{\sigma_{\lambda}}$. It is clear that $A(\lambda)\cong K_{m,n}/\mathcal{I}_{\lambda}$, where $\mathcal{I}_{\lambda}$ is the left ideal of $K_{m,n}$ generated by $d_i-\lambda_i, \frac{\partial}{\partial \xi_j},i=1,\ldots,m;j=1,\ldots,n$. 

\begin{lemma}\label{lem3.5} {\rm (1).}  $A(\lambda)$ is a strictly simple $K_{m,n}$ module;

{\rm (2).}  Any simple weight $K_{m,n}$ module is isomorphic to some $A(\lambda)$ for some $\lambda\in \bC^m$ up to a parity-change.
\end{lemma}
\begin{proof} It is easy to see that $A(0)$ hence $A(\lambda)$ is a strictly simple $K_{m,n}$ module.  Now let $V$ be any simple weight $K_{m,n}$ module with $\lambda\in {\rm supp}(V)$. Fix a nonzero homogeneous element $v\in V_{\lambda}$. Since $V'=\bC[\frac{\partial}{\partial \xi_1},\ldots,\frac{\partial}{\partial \xi_n}]v$ is a finite-dimensional supersubspace with  $\frac{\partial}{\partial \xi_1},\ldots,\frac{\partial}{\partial \xi_n}$ acting nilpotently, we may find a nonzero homogeneous element $v'\in  V'$ with $\mathcal{I}_{\lambda} v'=0$. Then up to a parity-change $V=K_{m,n}v'$ is isomorphic to a simple quotient of $A(\lambda)\cong K_{m,n}/\mathcal{I}_{\lambda}$. That is $V\cong A(\lambda)$. \end{proof}

Now for any $\mathfrak{m}\Delta$ module $V$, we have the $AW$ module  $\Gamma(\lambda,V ):=(A_{m,n}(\lambda)\otimes V)^{\phi_1}$, where  $\phi_1:\bar{U}\stackrel{\eta}{\longrightarrow} K_{m,n}\otimes U(\mathcal{T}) \stackrel{{\rm id}\otimes \psi }{\longrightarrow} K_{m,n}\otimes U(\mathfrak{m}\Delta)$. More precisely, $\Gamma(\lambda,V)=A_{m,n}\otimes V$ with actions
\begin{equation}\label{eq3.7}\begin{split}&t^{\beta}\xi_Id_i \cdot (y\otimes v)=\sum_{J\subsetneq I} (-1)^{\tau(J,I\setminus J)+|I\setminus J||y|}t^{\beta}\xi_Jy\otimes t^{\beta}\xi_{I\setminus J}d_i v+t^{\beta}\xi_I y\otimes (t^{\beta}-1)d_i  v\\&+(t^{\beta}\xi_I(d_i+\lambda_i)(y))\otimes v;\\
&t^{\beta}\xi_I\frac{\partial}{\partial \xi_j} \cdot (y\otimes v)=\sum_{J\subsetneq I} (-1)^{\tau(J,I\setminus J)+(|I\setminus J|+1)|y|}t^{\beta}\xi_Jy\otimes t^{\beta}\xi_{I\setminus J}\frac{\partial y}{\partial \xi_j} v\\
&+(-1)^{|y|}t^{\beta}\xi_I y\otimes (t^{\beta}-1)\frac{\partial}{\partial \xi_j}  v+t^{\beta}\xi_I\cdot \frac{\partial y}{\partial \xi_j} \otimes v;\\
&x\cdot (y\otimes v)=xy\otimes v,\forall x,y\in A, v\in V.
\end{split}
\end{equation}

\begin{lemma}\label{lem3.6} {\rm (1).} For any $\lambda\in \bC^m$ and any simple $\mathfrak{m}\Delta$ module $V$, $\Gamma(\lambda, V)$ is a simple weight $AW_{m,n}$ module.

{\rm (2).} Let $M$ be any simple weight $AW_{m,n}$ module with $\lambda\in {\rm supp}(M)$. There exists a simple $\mathfrak{m}\Delta$ module $V$ such that $M\cong \Gamma(\lambda, V)$. \end{lemma}
\begin{proof} From Lemma \ref{lem3.5} (1) and Lemma \ref{lem2.2} (3) , we know that $A(\lambda)\otimes V$ is a simple $K_{m,n}\otimes U(\mathfrak{m}\Delta)$ module for any $\lambda\in \bC^m$ and any simple $\mathfrak{m}\Delta$ module $V$. From the definition of $\Gamma(\lambda, V)$, we have (1).
Let $M$ be any simple weight $AW$ module with $\lambda\in {\rm supp}(V)$. Then $M^{\phi_1^{-1}}$ is a simple $K_{m,n}\otimes U(\mathfrak{m}\Delta)$ module. By a same argument as in the proof of Lemma \ref{lem3.5} (2), we may find a nonzero homogeneous $v'\in M_{\lambda}$ with $\mathcal{I}_{\lambda} v'=0$ and $K_{m,n}v'\cong A(\lambda)$ or $K_{m,n}v'\cong \Pi(A(\lambda))$. From Lemma \ref{lem2.2} (4), there exists a simple $U(\mathfrak{m}\Delta)$ module $P$ such that $M^{\phi_1^{-1}}\cong A(\lambda)\otimes P$ or $M^{\phi_1^{-1}}\cong  \Pi(A(\lambda))\otimes P\cong A(\lambda)\otimes \Pi(P^{T})$. Note that the last isomorphism is due to Lemma \ref{lem2.1}. Thus (2) follows.

 \end{proof}


\begin{lemma}\label{lem3.7}{\rm (1).} $\mathfrak{m}\Delta/\mathfrak{m}^2\Delta\cong \gl(m,n)$;

{\rm (2).} Let  $V$ be any finite-dimensional  $\mathfrak{m}\Delta$ module.  Then there exists some $k\in\N$ such that $\mathfrak{m}^k \Delta V=0$;

{\rm (3). } Let  $V$ be any finite-dimensional  simple $\mathfrak{m}\Delta$ module. Then we have $\mathfrak{m}^2\Delta V=0$. Therefore, $V$ can be regarded as a simple $\gl(m,n)$ module via the isomorphism in (1).\end{lemma}

\begin{proof} It is easy to verify that the linear map $\pi: \mathfrak{m}\Delta/\mathfrak{m}^2\Delta\rightarrow \gl(m,n)$ defined by \begin{align*}&\pi((t_i-1)\frac{\partial}{\partial t_j}+\mathfrak{m}^2\Delta)=E_{i,j}, \pi(\xi_s\frac{\partial}{\partial \xi_j}+\mathfrak{m}^2\Delta)=E_{m+s,m+j}, \\& \pi((t_i-1)\frac{\partial}{\partial \xi_j}+\mathfrak{m}^2\Delta)=E_{i,m+j}, \pi(\xi_j\frac{\partial}{\partial t_i}+\mathfrak{m}^2\Delta)=E_{m+j,i}.\end{align*} is a Lie superalgebra isomorphism. So we have (1).

 Let  $V$ be any finite-dimensional  $\mathfrak{m}\Delta$ module. Let $\Delta'={\rm span}\{\frac{\partial}{\partial t_i}, \frac{\partial}{\partial \xi_j}|  i=1,2,\ldots,m; j=1,2,\ldots,n\}$. Let $A^+=\bC[t_1,\ldots,t_m,\xi_1,\ldots,\xi_n]$, $d=\sum_{i=1}^{m}(t_i-1)\frac{\partial}{\partial t_i}+\sum_{j=1}^{n} \xi_j\frac{\partial}{\partial \xi_j}$ and $\mathfrak{m}^+=\mathfrak{m}\cap A^+$. Then $\mathfrak{m}^+=\oplus_{i=1}^{+\infty}\mathfrak{m}^+_i$ with $\mathfrak{m}^+_k=\{x\in \mathfrak{m}^+|[d,x]=kx\}={\rm span}\{(t_1-1)^{p_1}\cdots (t_m-1)^{p_m}\xi_I \in A^+|p_1+\cdots+p_m+|I|=k,p_1,\ldots,\p_m\in \Z_+\}.$ And $\mathfrak{m}^+\Delta'$ is a Lie supersubalgebra of $W=A\Delta'$.
 Let $f(\lambda)=\Pi_{i=1}^s(\lambda-\lambda_i)^{k_i}$ be the characteristic polynomial of $d$ as an operator on $V$.  Then there exists some integer number $k>3$ such that $(f(\lambda-l), f(\lambda))=1$ if $l\ge k-3$. From $f(d-l) (x v)=xf(d)v=0,\forall x\in \mathfrak{m}^+_{l+1}\Delta', v \in V$, we have $xv=0$. That is $(\mathfrak{m}^+)^{k-2} \Delta'V=0$. Let $\mathfrak{a}$ be the ideal of $\mathfrak{m}\Delta'$ generated by $(\mathfrak{m}^+)^{k-2}\Delta'$. Then $\mathfrak{a} V=0$. From $[yt^{\beta}d, xd]-[t^{\beta}d, yxd]=[y,xd]t^\beta d+y[t^{\beta}d,xd]-[t^{\beta}d,y]xd-y[t^{\beta}d,xd]=-2t^{\beta}yxd\in \mathfrak{a},\forall x\in  (\mathfrak{m}^+)^{k-2},y\in \mathfrak{m}^+_1,\beta\in\Z^m $ we have $\mathfrak{m}^{k-1}d\subseteq\mathfrak{a}$. So   $yx\partial=[yd, x\partial]+(-1)^{|x\partial||y|}x[\partial,y]d\in \mathfrak{a},\forall x\in \mathfrak{m}_1^+, y\in \mathfrak{m}^{k-1},\partial\in \Delta'$. Thus $\mathfrak{m}^k\Delta' \subseteq  \mathfrak{a}$, which implies (2).

 Now suppose $V$ is a finite-dimensional  simple $\mathfrak{m}\Delta$ module. From (2), $V$ is also a simple module over $\mathfrak{m}^+\Delta'/(\mathfrak{m}^+)^k\Delta'\cong \mathfrak{m}\Delta'/\mathfrak{m}^k\Delta'=\mathfrak{m}\Delta/\mathfrak{m}^k\Delta$ for some $k$.  So $d$ is diagonalizable on $V$. Let $\{\lambda_1,\ldots,\lambda_s\}$ be all eigenvalues of $d$ on $V$. Then all  eigenvalues of $d$ on $(\mathfrak{m}^+)^2\Delta'V$ is contained in $\{\lambda_1,\ldots,\lambda_s\}+\N$. Thus $(\mathfrak{m}^+)^2\Delta'V\ne V$.   From the simplicity of $V$ we have $(\mathfrak{m}^+)^2\Delta'V=0$. Hence $\mathfrak{m}^2\Delta V=0$.  So we have (3) hold.
\end{proof}

Now we have the Lie superalgebra homomorphism $\bar{\psi}:  \mathfrak{m}\Delta\rightarrow \mathfrak{m}\Delta/\mathfrak{m}^2\Delta\rightarrow  \gl(m,n)$ with
\begin{equation}\begin{split}&\bar{\psi}((t^{\alpha}-1)d_i)=\sum_{s=1}^m \alpha_s E_{si}, \bar{\psi}(t^{\alpha}\xi_Id_i)
=\left\{\begin{aligned}&E_{m+j,i}, &I=\{j\};\\
&0, &|I|>1,\end{aligned}\right.\\
 &\bar{\psi}(t^{\alpha}-1)\frac{\partial}{\partial \xi_j})=\sum_{s=1}^m \alpha_sE_{s,m+j},\bar{\psi}(t^{\alpha}\xi_I\frac{\partial}{\partial \xi_j})=\left\{\begin{aligned}&E_{m+s,m+j}, &I=\{s\};\\
&0, &|I|>1.\end{aligned}\right. \end{split}\end{equation}

We therefore have the associative superalgebra homomorphism $\phi: \bar{U}\rightarrow K_{m,n}\otimes U(\mathfrak{m}\Delta)\rightarrow K_{m,n}\otimes U(\gl(m,n))$ with
\begin{align*}&\phi(t^{\alpha}\xi_{I})=t^{\alpha}\xi_{I}\otimes 1,\\
& \phi(t^{\alpha}\xi_{I}d_i)=\sum_{s=1}^n (-1)^{|I|-1}\frac{\partial}{\partial \xi_s}(t^{\alpha}\xi_{I})\otimes E_{m+s,i}+\sum_{s=1}^m d_s(t^{\alpha}\xi_{I})\otimes E_{s,i}+(t^{\alpha}\xi_{I}\cdot d_i)\otimes 1,\\ &\phi(t^{\alpha}\xi_{I}\frac{\partial}{\partial \xi_j})=\sum_{s=1}^n (-1)^{|I|-1}\frac{\partial}{\partial \xi_s}(t^{\alpha}\xi_{I})\otimes E_{m+s,m+j}+\sum_{s=1}^m d_s(t^{\alpha}\xi_{I})\otimes E_{s,m+j}+(t^{\alpha}\xi_{I}\cdot \frac{\partial}{\partial \xi_j})\otimes 1.\end{align*}

For any  $\gl(m,n)$ module $V$, we have the $AW$ module $\Gamma(\lambda,V):=(A(\lambda)\otimes V)^{\phi}$, which will be  called a {\it tensor} module or {\it Shen-Larsson} module. More precisely, $\Gamma(\lambda,V)=A_{m,n}\otimes V$ with actions
\begin{equation}\label{eq3.9}\begin{split}&xd_i \cdot(y\otimes v)=\sum_{s=1}^n (-1)^{|x|-1+|y|}\frac{\partial x}{\partial \xi_s}y \otimes E_{m+s,i}v+ \sum_{s=1}^m d_s(x)y\otimes E_{s,i}v+x(d_i(y)+\lambda_iy)\otimes v;\\
&x\frac{\partial}{\partial \xi_j} \cdot (y\otimes v)=\sum_{s=1}^n (-1)^{|x|-1}\frac{\partial x}{\partial {\xi_s}}y\otimes E_{m+s,m+j}v+\sum_{s=1}^m (-1)^{|y|}d_s(x)y\otimes E_{s,m+j}v +x\frac{\partial y}{\partial \xi_j}\otimes v;\\
&x\cdot (y\otimes v)=xy\otimes v,\forall x,y\in A_{m,n}, v\in V.
\end{split}
\end{equation}

Recall that the finite-dimensional simple $\gl(m,n)$ modules were classified in \cite{K}. We classify the simple cuspidal $AW$ module in the following lemma.

\begin{lemma}\label{lem3.8} For any  simple cuspidal $AW_{m,n}$ module $M$, there exists some finite-dimensional simple $\gl(m,n)$ module $V$ and $\lambda\in \bC^m$ such that $M\cong \Gamma(\lambda, V)$. \end{lemma}

\begin{proof}Let $M$ be a   simple cuspidal $AW_{m,n}$ module. Then from Lemma \ref{lem3.6} (2),  $M\cong \Gamma(\lambda, V)$ for some $\lambda\in \bC^m$ and a simple $\mathfrak{m}\Delta$ module $V$. Since $M$ is cuspidal, we know that $V$ is a finite-dimensional $\mathfrak{m}\Delta$ module. Thus from Lemma \ref{lem3.7} (3), $V$ can be regarded as a simple $\gl(m,n)$ module. So we have proved the lemma.\end{proof}

Next we are going to define the $A$-cover $\hat{M}$ of a cuspida $W$ module $M$.

Consider $W$ as the adjoint $W$ module. We can make the tensor product $W$ module
 $W\otimes M$ into an $AW$ module by defining
$$ x \cdot (y \otimes v)=(xy)\otimes v, \forall x\in A, y\in W, v\in M. $$

Denote $K(M)=\{ \sum_{i=1}^k x_i\otimes v_i\in W\otimes M|  \sum_{i=1}^k(ax_i) v_i=0,\forall a\in A\}$. Then it is easy to see
$K(M)$ is a $AW$ submodule of $W\otimes M$.  Then we have the $AW$ module  $\hat{M}=(W\otimes M)/K(M)$. As in \cite{BF1},  we call $\hat{M}$ {\it the cover of} $M$ if $WM=M$.

Clearly,  the linear map
\begin{equation}\label{cover}\begin{split}
\pi: \,\,\hat{M}  &\to\ \ WM,\\
         w\otimes y+K(M)\ & \mapsto\ \  wy,\quad \forall\ w\in W, y\in M
\end{split}\end{equation}
is a $W$ module epimorphism.

\begin{lemma}\label{lem3.9}For any  cuspidal $W_{m,n}$ module $M$, there exists some $l_0\in \N$ such that
$\sum_{i=0}^{l_0}(-1)^i \binom {l_0} i t^{\alpha+i\gamma} \xi_I \partial \cdot t^{\beta-i\gamma} d_\mu v=0,\,\forall\,\partial\in \Delta, I\subseteq \{1,\ldots,n\}, v\in M,\alpha,\beta,\gamma\in \Z^m,\mu\in \bC^m$.
 \end{lemma}

 \begin{proof} We only need to prove it for generic $\mu$.  It's clear for $\gamma=0$. We assume that $\gamma\ne 0$ and $\mu$ is generic. Denote  ${\ann(M)}=\{x\in U(W)|x M=0\}$. From \cite{BF1}, there exists some $l_0\ge 3$ such that  $\Omega_{\alpha,\beta,\gamma}^{(l)} \in \ann(M), \forall l\ge l_0-3,\alpha,\beta\in \Z^m$, where $\Omega_{\alpha,\beta,\gamma}^{(l)}=\sum_{i=0}^{l}(-1)^i \binom {l} i t^{\alpha-i\gamma} d_\mu  \cdot t^{\beta+i\gamma} d_\mu$.  Then

  \begin{align*}&[\Omega_{\alpha+s\gamma,\beta+p\gamma,\gamma}^{(l)}, t^{k\gamma}\xi_I \frac{\partial}{\partial \xi_j}]\\ &=\sum_{i=0}^{l}(-1)^i \binom {l} i [t^{\alpha+(s-i)\gamma} d_\mu, t^{k\gamma}\xi_I \frac{\partial}{\partial \xi_j}]  \cdot t^{\beta+(p+i)\gamma} d_\mu \\ &+\sum_{i=0}^{l}(-1)^i \binom {l} i t^{\alpha+(s-i)\gamma} d_\mu \cdot [ t^{\beta+(p+i)\gamma} d_\mu ,t^{k\gamma}\xi_I\frac{\partial}{\partial \xi_j}] \\ &=k(\mu,\gamma)\sum_{i=0}^{l}(-1)^i \binom {l} i t^{\alpha+(s+k-i)\gamma}\xi_I \frac{\partial}{\partial \xi_j} \cdot t^{\beta+(p+i)\gamma} d_\mu \\ &+k(\mu,\gamma)\sum_{i=0}^{l}(-1)^i \binom {l} i   t^{\alpha+(s-i)\gamma} d_\mu \cdot  t^{\beta+(p+k+i)\gamma}\xi_I\frac{\partial}{\partial \xi_j}\in \ann(M).\end{align*}

   Thus \begin{align*}&2[\Omega_{\alpha+\gamma,\beta,\gamma}^{(l)}, t^{-\gamma}\xi_I \frac{\partial}{\partial \xi_j}]-[\Omega_{\alpha+\gamma,\beta+\gamma,\gamma}^{(l)}, t^{-2\gamma}\xi_I \frac{\partial}{\partial \xi_j}]\\ &=-2(\gamma,\mu)\Big( \sum_{i=0}^{l}(-1)^i \binom {l} i  t^{\alpha-i\gamma}\xi_I \frac{\partial}{\partial \xi_j} \cdot t^{\beta+i\gamma} d_\mu - \sum_{i=0}^{l}(-1)^i \binom {l} i  t^{\alpha-(1+i)\gamma}\xi_I \frac{\partial}{\partial \xi_j} \cdot t^{\beta+(i+1)\gamma} d_\mu \Big)\\&=-2(\gamma,\mu)\Big( \sum_{i=0}^{l+1}(-1)^i \binom {l+1} i t^{\alpha-i\gamma} \xi_I \frac{\partial}{\partial \xi_j}\cdot t^{\beta+i\gamma} d_\mu \Big)\in \ann(M).\end{align*} We get  \begin{equation}\label{eq3.11}\sum_{i=0}^{l+1}(-1)^i \binom {l+1} i t^{\alpha-i\gamma} \xi_I \frac{\partial}{\partial \xi_j}\cdot t^{\beta+i\gamma} d_\mu \in \ann(M).\end{equation}

From  \begin{align*}&[\Omega_{\alpha+s\gamma,\beta+p\gamma,\gamma}^{(l)}, t^{k\gamma}\xi_I d_{\mu}]\\ &=\sum_{i=0}^{l}(-1)^i \binom {l} i [t^{\alpha+(s-i)\gamma} d_\mu, t^{k\gamma}\xi_I d_\mu]  \cdot t^{\beta+(p+i)\gamma} d_\mu
\\ &+\sum_{i=0}^{l}(-1)^i \binom {l} i t^{\alpha+(s-i)\gamma} d_\mu \cdot [ t^{\beta+(p+i)\gamma} d_\mu ,t^{k\gamma}\xi_Id_\mu] \\ &=\sum_{i=0}^{l}(-1)^i \binom {l} i (\mu,-\alpha+(k-s+i)\gamma)t^{\alpha+(s+k-i)\gamma}\xi_I d_\mu \cdot t^{\beta+(p+i)\gamma} d_\mu \\ &+\sum_{i=0}^{l}(-1)^i \binom {l} i  (\mu,-\beta+(k-p-i)\gamma) t^{\alpha+(s-i)\gamma} d_\mu \cdot  t^{\beta+(p+k+i)\gamma}\xi_Id_\mu \in \ann(M),\end{align*}

We have

\begin{align*}&f(s,p,k)\\&:=2[\Omega_{\alpha+s\gamma,\beta+p\gamma,\gamma}^{(l)},t^{k\gamma}\xi_I d_{\mu}]-[\Omega_{\alpha+s\gamma,\beta+(p+1)\gamma,\gamma}^{(l)},t^{(k-1)\gamma}\xi_I d_{\mu}]-[\Omega_{\alpha+s\gamma,\beta+(p-1)\gamma,\gamma}^{(l)}, t^{(k+1)\gamma}\xi_I d_{\mu}]\\ &=2\sum_{i=0}^{l}(-1)^i \binom {l} i (\mu,-\alpha+(k-s+i)\gamma) t^{\alpha+(s+k-i)\gamma}\xi_I d_\mu \cdot t^{\beta+(p+i)\gamma} d_\mu\\&-
\sum_{i=0}^{l}(-1)^i \binom {l} i (\mu,-\alpha+(k-1-s+i)\gamma)t^{\alpha+(s+k-1-i)\gamma}\xi_I d_\mu \cdot t^{\beta+(p+1+i)\gamma} d_\mu \\&-
\sum_{i=0}^{l}(-1)^i \binom {l} i (\mu,-\alpha+(k+1-s+i)\gamma)t^{\alpha+(s+k+1-i)\gamma}\xi_I d_\mu \cdot t^{\beta+(p-1+i)\gamma} d_\mu\in \ann(M).
\end{align*}

Then \begin{equation}\label{eq3.12}\begin{split}& \frac 1{-2(\mu,\gamma)}(f(0,1,-1)-f(1,1,-2))\\&= -2\sum_{i=0}^{l}(-1)^i \binom {l} i t^{\alpha-(1+i)\gamma}\xi_I d_\mu \cdot t^{\beta+(1+i)\gamma} d_\mu+\sum_{i=0}^{l}(-1)^i \binom {l} i t^{\alpha-(2+i)\gamma}\xi_I d_\mu \cdot t^{\beta+(2+i)\gamma} d_\mu\\ &+\sum_{i=0}^{l}(-1)^i \binom {l} i t^{\alpha-i\gamma}\xi_I d_\mu \cdot t^{\beta+i\gamma} d_\mu\\&
=\sum_{i=0}^{l+2}(-1)^i \binom {l+2} i t^{\alpha-i\gamma}\xi_I d_\mu \cdot t^{\beta+i\gamma} d_\mu\in \ann(M).\end{split}
\end{equation}
Similarly, we have $\sum_{i=0}^{l+2}(-1)^i(^{l+2}_i)t^{\alpha-i\gamma}d_\mu\cdot t^{\beta+i\gamma}\xi_Id_\mu\in\ann(M)$.

Now for any $\mu'\in \bC^m$ with $(\mu',\gamma)=0$, we have
 \begin{align*}&[\Omega_{\alpha+s\gamma,\beta+p\gamma,\gamma}^{(l)}, t^{k\gamma}\xi_I d_{\mu'}]\\ &=\sum_{i=0}^{l}(-1)^i \binom {l} i [t^{\alpha+(s-i)\gamma} d_\mu, t^{k\gamma}\xi_I d_{\mu'}]  \cdot t^{\beta+(p+i)\gamma} d_\mu \\ &+\sum_{i=0}^{l}(-1)^i \binom {l} i t^{\alpha+(s-i)\gamma} d_\mu \cdot [ t^{\beta+(p+i)\gamma} d_\mu ,t^{k\gamma}\xi_Id_{\mu'}] \\&=k(\mu,\gamma)\sum_{i=0}^{l}(-1)^i \binom {l} i t^{\alpha+(s+k-i)\gamma}\xi_I d_{\mu'} \cdot t^{\beta+(p+i)\gamma} d_\mu\\&-(\alpha,\mu')\sum_{i=0}^{l}(-1)^i \binom {l} i t^{\alpha+(s+k-i)\gamma}\xi_I d_\mu \cdot t^{\beta+(p+i)\gamma} d_\mu \\ &+k(\mu,\gamma)\sum_{i=0}^{l}(-1)^i \binom {l} i t^{\alpha+(s-i)\gamma} d_\mu \cdot t^{\beta+(k+p+i)\gamma} \xi_Id_{\mu'}\\&-(\beta,\mu')\sum_{i=0}^{l}(-1)^i \binom {l} i t^{\alpha+(s-i)\gamma} d_\mu \cdot  t^{\beta+(k+p+i)\gamma}\xi_I d_\mu\in \ann(M).\end{align*} 
 From (\ref{eq3.12}), we have $g(s,p,k):=\sum_{i=0}^{l+2}(-1)^i \binom {l+2} i t^{\alpha+(s+k-i)\gamma}\xi_I d_{\mu'} \cdot t^{\beta+(p+i)\gamma} d_\mu\newline+\sum_{i=0}^{l+2}(-1)^i \binom {l+2} i t^{\alpha+(s-i)\gamma} d_\mu \cdot t^{\beta+(k+p+i)\gamma} \xi_Id_{\mu'}\in \ann(M),\forall k\ne 0$.

Then \begin{equation}\label{eq3.13}\begin{split}&g(1,0,-1)-g(1,1,-2)\\&=\sum_{i=0}^{l+2}(-1)^i \binom {l+2} i t^{\alpha-i\gamma}\xi_I d_{\mu'}\cdot t^{\beta+i\gamma} d_\mu\\& -\sum_{i=0}^{l+2}(-1)^i \binom {l+2} i t^{\alpha-(i+1)\gamma}\xi_I d_{\mu'} \cdot t^{\beta+(1+i)\gamma} d_\mu\\&=\sum_{i=0}^{l+3}(-1)^i \binom {l+3} i t^{\alpha-i\gamma}\xi_I d_{\mu'} \cdot t^{\beta+i\gamma} d_\mu\in \ann(M).\end{split}\end{equation}

The lemma follows from (\ref{eq3.11}),(\ref{eq3.12}) and (\ref{eq3.13}).
 \end{proof}

\begin{lemma}\label{lem3.10} For any  cuspidal $W_{m,n}$ module $M$, $\hat{M}$ is also  cuspidal.\end{lemma}

\begin{proof} It is obvious for $m=0$. Suppose that $m\in \N$. Let $\|\alpha||= \sum_{i=1}^m|\alpha_i|$ for all $\alpha\in \Z^m$.  From Lemma \ref{lem3.9}, there exists some $l_0\in \N$ such that
$\sum_{i=0}^{l_0}(-1)^i \binom {l_0} i t^{\alpha-i\gamma} \xi_I \partial \cdot t^{\beta+i\gamma} d_\mu v=0,\,\forall\,\partial\in \Delta, I\subseteq \{1,\ldots,n\}, v\in M,\alpha,\beta,\gamma\in \Z^m,\mu\in\bC^m$.  Then

\begin{equation}\label{eq3.14}\sum_{i=0}^{l_0}(-1)^i \binom {l_0} i t^{\alpha-i\gamma} \xi_I \partial \otimes t^{\beta+i\gamma} d_\mu v\in K(M) \end{equation}
for all $\partial\in \Delta, I\subseteq \{1,\ldots,n\}, v\in M,\alpha,\beta,\gamma\in \Z^m,\mu\in\bC^m$.

We are going to prove by induction on $\|\alpha\|$ that
\begin{equation}\label{eq3.15}t^{\alpha}\xi_I \partial \otimes t^{\beta}d_{\mu} v\in \sum_{\begin{matrix}\|\alpha' \|\le ml_0,\\ I'\subseteq \{1,\ldots,n\}\end{matrix}}t^{\alpha'}\xi_{I'}\Delta\otimes M +K(M),\end{equation} for all $\alpha,\beta\in \Z^m,\partial\in \Delta, I\subseteq  \{1,\ldots,n\},\mu\in\bC^m,v\in M_{\lambda}$.
This is obvious for $\alpha\in\Z^m$ with $\|\alpha\|\le ml_0$. Now we assume that $\|\alpha\|>ml_0$. Without lose of generality, we may assume that $\alpha_1>l_0$. Then by  (\ref{eq3.14}) and  the induction hypothesis, for any $j\in\{1,\dots,m\}$, we have

$t^{\alpha}\xi_I \partial \otimes t^{\beta}d_j v=(\sum_{i=0}^{l_0}(-1)^i \binom {l_0} i t^{\alpha-ie_1} \xi_I \partial \otimes t^{\beta+ie_1} d_j v)-(\sum_{i=1}^{l_0}(-1)^i \binom {l_0} i t^{\alpha-ie_1} \xi_I \partial \otimes t^{\beta+ie_1} d_j v)\in K(M)$ as desired.

Since $M=M_0+(\sum_{j=1}^md_jM)$, we deduce that
 \begin{equation}W\otimes M=
W\otimes M_0+\sum_{\begin{matrix}\|\alpha \|\le ml_0,\\ I\subseteq \{1,\ldots,n\}\end{matrix}}t^{\alpha}\xi_I\Delta\otimes M +K(M).\end{equation} Now it's clear that $\hat{M}=(W\otimes M)/K(M)$ is  cuspidal.
\end{proof}

\begin{theorem}\label{the3.11}Let $(m,n)\in \Z_+^2\backslash\{(0,0)\}$. Any nontrivial  simple cuspidal $W_{m,n}$ module is isomorphic to a simple quotient of a tensor module $\Gamma(\alpha, V)$ for some finite-dimensional simple $\gl(m,n)$ module $V$ and some $\alpha\in \bC^m$.  \end{theorem}
\begin{proof} Let $M$ be any nontrivial  simple cuspidal $W$ module. Then $WM=M$, and there is an epimorphism $\pi:\hat{M}\rightarrow M$. From Lemma \ref{lem3.10}, $\hat{M}$ is  cuspidal. Hence $\hat{M}$ has a composition series of $AW$ submodules:
$$0=\hat{M}^{(1)}\subset \hat{M}^{(2)}\subset\cdots \subset \hat{M}^{(s)}=\hat{M}$$ with $\hat{M}^{(i)}/\hat{M}^{(i-1)}$ being simple $AW$ modules. Let $l$ be the minimal integer such that $\pi(\hat{M}^{(l)})\ne 0$.  Since $M$ is simple $W$ module, we have $\pi(\hat{M}^{(l)})=M$ and $\pi(\hat{M}^{(l-1)})=0$. This gives us an epimorphism of $W$ modules from the simple cuspidal $AW$ module $\hat{M}^{(l)}/\hat{M}^{(l-1)}$ to $M$. From
 Lemma \ref{lem3.8}, we have $\hat{M}^{(l)}/\hat{M}^{(l-1)}$  is isomorphic to a tensor module  $\Gamma(\alpha,V)$  for  a finite-dimensional simple $\gl(m,n)$ module $V$ and an $\alpha\in \bC^m$. This completes the proof.
\end{proof}

\section{ Simple cuspidal modules over the extended Witt superalgebras}

For any $m\in \N$, let $G=\Z e_1+\cdots \Z e_{m-1}$, $W_{m,n}':=(W_{m,n})_{G}=A_{m-1,n}\Delta=W_{m-1,n}+A_{m-1,n}d_m, \tilde{W}_{m,n}':=(\tilde{W}_{m,n})_G=W_{m,n}'+A_{m-1,n}$ be Lie supersubalgebras of $\tilde{W}_{m,n}$. Clearly we have   $W_{m,n}'\cong \tilde{W}_{m-1,n}$. In this section, we determine the  simple cuspidal module over $W_{m,n}'$.  Similarly, a $\tilde{W}_{m,n}'$ module is called $AW_{m,n}'$ module if $A_{m-1,n}$ acts associatively.

 Let  $\mathcal{J}'$ be the left ideal of $U(\tilde{W}_{m,n}')$ generated by $\{t^0-1,t^{\alpha}\xi_I\cdot t^{\beta}\xi_J-t^{\alpha+\beta}\xi_I\xi_J|\alpha,\beta\in \Z^{m-1}, I,J\subseteq \{1,2,\ldots,n\} \}$. Then $\mathcal{J}'$ is  an ideal of $U(\tilde{W}_{m,n}')$ and we have the quotient algebra $\bar{U}'=(U(A_{m-1,n}) U(W_{m,n}'))/\mathcal{J}'$. From PBW Theorem,  we may identify $A_{m-1,n}$, $W_{m,n}'$ with their images in $\bar{U}'$. Thus $\bar{U}'=A_{m-1,n}\cdot U(W_{m,n}')$ and we may regard $\bar{U}'$ as a supersubalgebra of $\bar{U}_{m,n}$. Let $\mathcal{T}'=\mathcal{T}_{m,n}\cap \bar{U}'$, $K'=K_{m,n}\cap  \bar{U}'$. Let ${\mathfrak{a}}_{m,n}={\rm Span}\{E_{s,t}|s,t\in \{1,2,\ldots,m-1, m+1,\ldots, m+n\}\}\subseteq \gl(m,n)$, ${\mathfrak{b}}_{m,n}={\rm Span}\{E_{s,m}|s\in \{1,2,\ldots,m-1, m+1,\ldots, m+n\}\}$ be Lie supersubalgebras of $\gl(m,n)$. Then

\begin{lemma}\label{lem4.1}  {\rm (1).} We have the  associative superalgebra isomorphism
\begin{equation}\label{iota1}
\iota':K' \otimes U(\mathcal{T}')\rightarrow \bar{U}', \  \
\iota'(x \otimes y)=x\cdot y
\end{equation}
where $x\in K', y\in U(\mathcal{T}')$.

{\rm (2).} $\mathcal{T}'\cong \mathfrak{m}_{m-1,n}\Delta_{m,n}$.

{\rm (3).} $ \mathfrak{m}_{m-1,n}\Delta_{m,n}/ \mathfrak{m}_{m-1,n}^2\Delta_{m,n}\cong {\mathfrak{a}}_{m,n}\ltimes {\mathfrak{b}}_{m,n}\cong \gl(m-1,n)\ltimes \bC^{m-1,n}$.

{\rm (4).} For any finite-dimensional simple $\mathfrak{m}_{m-1,n}\Delta_{m,n}$ module $V$, we have $(\mathfrak{m}_{m-1,n}^2 \Delta_{m,n}+\mathfrak{m}_{m-1,n}d_m) V=0$. Hence $V$ can be regarded as a simple $ \gl(m-1,n)$ module via $$\mathfrak{m}_{m-1,n}\Delta_{m,n}/(\mathfrak{m}_{m-1,n}^2 \Delta_{m,n}+\mathfrak{m}_{m-1,n}d_m)\cong \gl(m-1,n).$$
\end{lemma}

\begin{proof} The isomorphisms in (1)-(3) are restrictions of isomorphisms in Lemma \ref{lem3.3}, Theorem \ref{the3.4} and Lemma \ref{lem3.7} (1) respectively.  For any finite-dimensional simple $\mathcal{T}'$ module $V$, by a same argument as in  the proof of Lemma \ref{lem3.7} (2) and (3), we have $\mathfrak{m}_{m-1,n}^2 \Delta_{m,n} V=0$. Hence $V$ can be regarded as a finite-dimensional simple $\gl(m-1,n)\ltimes \bC^{m-1,n}$ module via the isomorphism in (3). Now we only need to show that $\bC^{m-1,n} V=0$.  Let $I$ be the identity matrix in $\gl(m-1,n)$. Let $\{\lambda_1,\ldots,\lambda_s\}$ be the set of all eigenvalues of $I$ on $V$. Then  from $[I,x]=x, \forall x\in  \bC^{m-1,n}$, we have the eigenvalues of $I$ on $\bC^{m-1,n} V$ are contained in  $\{\lambda_1,\ldots,\lambda_s\}+\N$. Therefore, $\bC^{m-1,n} V\ne V$. And from the simplicity of $V$ we deduce that $\bC^{m-1,n} V=0$ as required.\end{proof}

Now we have the homomorphism $\phi': \bar{U}'\rightarrow K'\otimes U({\mathfrak{m}}_{m-1,n}\Delta_{m,n})\rightarrow K'\otimes U(\gl(m-1,n))$.
Let $A_{m,n}(\lambda)'=A_{m-1,n}\subset A_{m,n}(\lambda)$ be the $K'$ submodule of $A_{m,n}(\lambda)$. Then for any $\lambda\in \bC^m$ and any simple $\gl(m-1,n)$ module $V$ we have the $AW_{m,n}'$ module $\Gamma(\lambda,V)=(A_{m-1,n}(\lambda)' \otimes V)^{\phi'}$.
More precisely, $\Gamma(\lambda,V)=A_{m-1,n}\otimes V$ with the actions
 \begin{equation}\label{eq4.2}\begin{split}&xd_i \cdot(y\otimes v)=\sum_{s=1}^n (-1)^{|x|-1+|y|}\frac{\partial x}{\partial \xi_s} y\otimes E_{m-1+s,i}v+ \sum_{s=1}^{m-1} d_s(x)y\otimes E_{s,i}v+x(d_i+\lambda_i)(y)\otimes v;\\
&x\frac{\partial}{\partial \xi_j}\cdot  (y\otimes v)=\sum_{s=1}^n (-1)^{|x|-1}\frac{\partial x}{\partial {\xi_s}}y\otimes E_{m-1+s,m-1+j}v+\sum_{s=1}^{m-1} (-1)^{|y|}d_s(x)y\otimes E_{s,m-1+j}v\\& +x\frac{\partial y}{\partial \xi_j}\otimes v;\\
&xd_m\cdot (y\otimes v)=\lambda_mx y\otimes v;\\
&x\cdot (y\otimes v)=xy\otimes v,\forall x,y\in A_{m-1,n}, v\in V,i=1,2,\ldots,m-1;j=1,2,\ldots,n.
\end{split}
\end{equation}

\begin{lemma}\label{lem4.2} For any  simple cuspidal $AW_{m,n}'$ module $M$, there exists some finite-dimensional simple $\gl(m-1,n)$ module $V$ and $\lambda\in \bC^m$ such that $M\cong \Gamma(\lambda, V)$.\end{lemma}
\begin{proof} By a similar argument as in Lemma \ref{lem3.5}, we have $A_{m,n}(\lambda)'$ is a strictly simple $K'$ module and any simple weight $K'$ module is isomorphic to $A_{m,n}(\lambda)'$ up to a parity-change. Then by a similar argument  as in Lemma \ref{lem3.6}(2), we have any simple cuspidal $AW_{m,n}'$ module $M$ is isomorphic to some $(A_{m,n}(\lambda)'\otimes V)^{\iota'^{-1}}$ for some simple $\mathcal{T}'$ module $V$. Since $M$ is cuspidal, we have $V$ is finite-dimensional. Thus from Lemma \ref{lem4.1} (4) $V$ can be regarded as a simple $\gl(m-1,n)$ module.
The lemma follows from the definition of $\Gamma(\lambda, V)$.\end{proof}

Now we are going to define the $A_{m-1,n}$-cover $\hat{M}$ of a simple nontrivial  cuspida $W_{m,n}'$ module $M$.

Consider $W_{m,n}'$ as the adjoint $W_{m,n}'$ module. We can make the tensor product $W_{m,n}'$ module
 $W_{m,n}'\otimes M$ into an $AW_{m,n}'$ module by defining
$$ x \cdot (y \otimes v)=(xy)\otimes v, \forall x\in A_{m-1,n}, y\in W_{m,n}',v\in M. $$

Denote $K(M)=\{ \sum_{i=1}^k x_i\otimes v_i\in W_{m,n}'\otimes M|  \sum_{i=1}^k(ax_i) v_i=0,\forall a\in A_{m-1,n}\}$. Then it is easy to see
$K(M)$ is an $AW_{m,n}'$ submodule of $W_{m,n}'\otimes M$. Then we have  the  $AW_{m,n}'$ module
 $\hat{M}=(W_{m,n}'\otimes M)/K(M)$, which is called as {\it the cover of} $M$ if $W_{m,n}'M=M$.

 \begin{lemma}\label{lem4.3}If $M$ is a nontrivial simple cuspidal $W_{m,n}'$ module, then $\hat{M}$ is also cuspidal .\end{lemma}
\begin{proof}Let $L=\sum_{i=1}^{m-1} A_{m-1,n}(d_i+d_m)+\sum_{j=1}^n A_{m-1,n}\frac{\partial}{\partial \xi_j}$.
 Then it is clear that
 $L\cong W_{m-1,n}$. From $[d_m, W_{m,n}']=0$ and $M$ is a simple weight module, we see that $d_m$ acts as a scalar on $M$. Therefore, $M$ is a cuspidal $L$ module and
 a cuspidal $W_{m-1,n}$ module. Let $K_1(M)=\{ \sum_{i=1}^k x_i\otimes v_i\in W_{m-1,n}\otimes M|  \sum_{i=1}^k(ax_i) v_i=0,\forall a\in A_{m-1,n}\}$ and
 $K_2(M)=\{ \sum_{i=1}^k x_i\otimes v_i\in L \otimes M|  \sum_{i=1}^k(ax_i) v_i=0,\forall a\in A_{m-1,n}\}$. Then applying Lemma \ref{lem3.10} to $W_{m-1,n}$
  module $M$ and $L$ module $M$ respectively,  we have that $(W_{m-1,n}\otimes M)/ K_1(M)$ and $(L\otimes M)/ K_2(M)$ are cuspidal.  Therefore, there exists a $N\in \N$, such that $\dim (W_{m-1,n}\otimes M)_{\lambda}/ K_1(M)_{\lambda}\le N$ and $\dim (L\otimes M)_{\lambda}/ K_2(M)_{\lambda}\le N$ for all $\lambda\in {\rm supp}(\hat{M})$. Note that $K(M)\supseteq K_1(M)+K_2(M)$. Thus $\dim \hat{M}_{\lambda} =\dim (W_{m,n}'\otimes M)_{\lambda}/K(M)_{\lambda}\le \dim ((W_{m-1,n}+L)\otimes M)/(K_1(M)+K_2(M))_{\lambda}=\dim ((W_{m-1,n}\otimes M+K_2(M))/((K_1(M)+K_2(M))_{\lambda}+(L\otimes M+K_1(M))/((K_1(M)+K_2(M)))_{\lambda} \le \dim (W_{m-1,n}\otimes M)_{\lambda}/ K_1(M)_{\lambda}+\dim (L\otimes M)_{\lambda}/ K_2(M)_{\lambda}\le N+N=2N$ for all $\lambda\in {\rm supp}(\hat{M})$. The lemma follows. \end{proof}

\begin{theorem}\label{the4.4} Let $(m,n)\in \N\times \Z_+$.  Any nontrivial  simple cuspidal $W_{m,n}'$ module is isomorphic to a simple quotient of a tensor module $\Gamma(\alpha, V)$ for  a finite-dimensional simple $\gl(m-1,n)$ module $V$ and an $\alpha\in \bC^m$. \end{theorem}
\begin{proof}The proof is similar as that of Theorem \ref{the3.11}.\end{proof}


\section{Main result}

In this section, we are going to classification the simple  weight $W$ modules with finite-dimensional weight spaces which are not  cuspidal. 
 Let  $m,n\in \N$, $e_1,\dots,e_m$ be the standard basis of $\Z^m$ and $W=W_{m,n}$. Then $\dim W_\alpha=2^n(m+n),\forall \alpha\in \Z^m$.

\begin{lemma}\label{lem5.1}
Let $\alpha,\beta\in\Z^m$ with $\beta\neq 0$, then $[W_\alpha,W_\beta]=W_{\alpha+\beta}$.
\end{lemma}
\begin{proof}
Let $i'\in\{1,\dots,m\}$ such that $\beta_{i'}\neq 0$. For any $i\in\{1,\dots,m\},j\in\{1,\dots,n\},\ k\in\Z_+,I\subseteq \{1,\dots,n\}$, we have
\begin{equation*}\begin{split}
&[t^\alpha d_{i'},t^\beta\xi_I\frac{\partial}{\partial\xi_j}]=\beta_{i'}t^{\alpha+\beta}\xi_I\frac{\partial}{\partial\xi_j},\\
&[t^\alpha\xi_{I}\frac{\partial}{\partial\xi_1},t^\beta\xi_1 d_i]=t^{\alpha+\beta}\xi_{I}d_i+(-1)^{|I|}\alpha_it^{\alpha+\beta}\xi_1\xi_I\frac{\partial}{\partial\xi_1}.
\end{split}\end{equation*}
So $t^{\alpha+\beta}\xi_I\frac{\partial}{\partial\xi_j},t^{\alpha+\beta}\xi_Id_i\in[W_\alpha,W_\beta]$. Hence the lemma holds.
\end{proof}

From now on, we will assume that $M$ is a simple weight $W$ module with finite-dimensional weight spaces which is not  cuspidal. Let  $\lambda$ be a nonzero weight of $M$.


Now let $G$ be a subgroup of $\Z^m$, $\beta$ be a nonzero element of $\Z^m$ with $\Z^m=G\oplus\Z\beta$. Then $W$ has a triangular decomposition $W=W_{G-\N\beta}\oplus W_G\oplus W_{G+\N\beta}$.
Suppose $X$ is a simple weight $W_G$ module. Turn $X$ into a $W_G\oplus W_{G+\N\beta}$ module by setting $W_{G+\N\beta}\cdot X=0$.
Let $M(G,\beta,X)=U(W)\bigotimes_{U(W_G\oplus W_{G+\N\beta})}X$ be the induced $ W$ module. Then $M(G,\beta,X)$ has a unique simple quotient, which will be denoted by $L(G,\beta,X)$.   We call $L(G,\beta,X)$ a {\it module of highest weight type} if $X$ is a cuspidal $g_G$ module.

The following result is well-known.

\begin{lemma}\label{lem5.2}Let $M$ be a  $W_{1.0}$ weight module with finite-dimensional weight spaces and  ${\rm supp}(M)\subseteq \lambda+\Z$. If for any $v\in M$, there exists some $N(v)\in \N$ such that $t_1^i d_1 v=0,\forall i\ge N(v)$. Then the weight set of $M$ is upper bounded, i.e., ${\rm supp}(M)\subseteq \lambda_0-\Z_+$ for some  $\lambda_0\in \bC$. \end{lemma}



\begin{lemma}\label{lem5.3}If $m=1$, then $M$ is a highest (or lowest) weight module.\end{lemma}

\begin{proof}Since $M$ is not cuspidal, there is a $k\in \Z$, such that $\dim M_{-k+\lambda}>2^n(n+1)(\dim M_{\lambda}+\dim M_{\lambda-1})$. Without  lost of generality, we may assume that $k\in \N$. Then there exists a nonzero homogeneous element $w\in M_{-k+\lambda}$ such that $ W_{k}w= W_{k+1}w=0$. Hence from Lemma \ref{lem5.1}, $ W_i w=0,\forall i\ge k^2$.

It is easy to see that $M'=\{v\in M| \dim  W_{\N} v<\infty\}$ is a $ W$ submodule of $M$ with $w\in M'$. Hence $M'=M$. From Lemma \ref{lem5.2}, we know the weight set of $M$ is upper bounded. So $M$ has to be a highest weight module.
\end{proof}

We  regard the general linear group ${\rm GL}_m(\Z)$  as a subgroup of the automorphism group of $W_{m,n}$ by $B(t^{\alpha}\xi_Id_i)=t^{\alpha B^T}\xi_Id_{e_i B^{-1}}$ and $B(t^{\alpha}\xi_I\frac{\partial}{\partial \xi_j})=t^{\alpha B^T}\xi_I\frac{\partial}{\partial \xi_j}$ for all $i=1,2,\ldots,m;j=1,2,\ldots,n;\alpha\in\Z^m; I\subseteq \{1,2,\ldots,n\}, B\in {\rm GL}_m(\Z)$.

Let generic $\mu$ and $Vir[\mu]$ be as defined in Section 2.
\begin{lemma}\label{lem5.4}Suppose that $m>1$. Let $G$ be a subgroup of $\Z^m$, $\beta$ be a nonzero element of $\Z^m$ with $\Z^m=G\oplus\Z\beta$, $X$ be a simple weight  $ W_G$ module. Then $L(G,\beta, X)$ has finite-dimensional weight spaces if and only if $X$ is a cuspidal $ W_G$ module. \end{lemma}

\begin{proof}Replace $L(G,\beta, X)$ by $L(G,\beta, X)^g$ by a suitable $g\in {\rm GL}_m(\Z)$ if necessary, we may assume that $G=\Z e_1+\cdots+\Z e_{m-1}$ and that $\beta=e_m$. It is straightforward to verify that $ W$ is a $G$-extragraded exp-polynomial Lie superalgebra. Suppose that $X$ is a nontrivial cuspidal $ W_G$ module. From Theorem \ref{the4.4},   $X$ is a $G$-graded exp-polynomial $ W_G$ module. Then from Lemma \ref{exp}, $L(G,\beta, X)$ has finite-dimensional weight spaces. Now suppose that $L(G,\beta, X)$ has finite dimensional weight spaces. Note that the support set of $L(G,\beta,X)$ is contained in $\gamma+G-\Z_+ \beta$ for some $\gamma\in{\rm supp}(X)$. Hence any nontrivial simple $\Vir[\mu]_G$ sub-quotient of $L(G,\beta,X)$ has to be isomorphic to $L_{\Vir[\mu]}(G,\beta,X')$ for some simple cuspidal  $\Vir[\mu]_G$ module $X'$. This implies that all simple  $\Vir[\mu]_G$ sub-quotients of $X$ are cuspidal. So we have $X$ is cuspidal.
\end{proof}

Now we will follow \cite{MZ} to deal with the case $m>1$. As we will see, the proof in \cite{MZ} still works for our algebra $ W$.

\begin{lemma}\cite[Lemma 3.2]{MZ}\label{lem5.5} Suppose that $m>1$.
After an appropriate change of variables $t_1,t_2,\ldots,t_m$ and weight $\lambda$, we may assume that $\lambda\ne 0$ and that there is a nonzero homogeneous vector $w\in V_\lambda$ such that $g_{e_i}\cdot w=0,i=1,\dots,m$.
\end{lemma}
\begin{proof}
$M$ as a $Vir[\mu]$ module  is a weight module but not cuspidal. From \cite[Theorem 3.9]{LZ2}, we know that every nontrivial cuspidal $Vir[\mu]$ module has support set $\gamma+\Z^m$ or $\Z^m\setminus\{0\}$ for some $\gamma \in\bC^m$. So $V$ has a simple $Vir[\mu]$-subquotient $X$ that is not cuspidal. Again by \cite[Theorem 3.9]{LZ2}, $X$ is a $Vir[\mu]$ module of highest weight type, and after an appropriate change of variables $t_1,t_2,\ldots,t_m$, $X$ is isomorphic to $L_{Vir[\mu]}(G,e_1,Y)$, where $G$ is the subgroup of $\Z^m$ generated by $e_2,\dots,e_m$, $Y$ is a simple cuspidal $\Vir[\mu]_G$ module. From \cite{LZ2}, $\dim V_{-ke_1+\lambda},k\in\N$ are not uniformly bounded. Fix an integer $N>3$, and let $B_N(\lambda)=\lambda+\{\alpha\in\Z^m\big||\alpha_i|\leqslant N,i=1,\dots,m\}$. Since $B_N(\lambda)$ is a finite set, there is a $k\in\N$ such that$-ke_1+\lambda\ne 0$ and $\dim M_{-ke_1+\lambda}>2^n(m+n)\sum_{\beta\in B_N(\lambda)}\dim M_\beta$.


Set $e'_1=(k+1)e_1+e_2,e'_2=ke_1+e_2,e'_j=e'_1+e_j$ for $3\leqslant j\leqslant m$. Then $\{e'_1,\dots,e'_m\}$ is a new $\Z$-basis of $\Z^m$ and $-ke_1+\lambda+e'_j\in B_N(\lambda)$ for all $j$. Since $\dim \g_{e'_i}=2^n(m+n)$, there exists a  nonzero homogeneous element $w\in V_{-ke_1+\lambda}$ such that $g_{e'_i}\cdot w=0$ for all $i$. Then the lemma follows after replacing $\lambda$ with $-ke_1+\lambda$ and $e_i$ with $e'_i$.\end{proof}



\begin{lemma}\cite[Lemma 3.3]{MZ}\label{lem5.6}
Suppose $m>1$. Let $e_1,\dots,e_m,\lambda$ and $w$ be as in Lemma \ref{lem5.5}. Then for any $v\in M$ there is a  $N(v)\in \N$ such that $ W_\alpha\cdot v=0$ for any $\alpha\in\Z^m$ with $\alpha_i>N(v), \forall i=1,\dots,m$.
\end{lemma}
\begin{proof} Let $M'$ be the subset of $M$ consisting of $v\in M$ for which there is a  $N(v)\in \N$ such that $ W_\alpha\cdot v=0$ for any $\alpha\in\Z^m$ with $\alpha_i>N(v), \forall i=1,\dots,m$. We need to show that $M=M'$. Clearly $M'$ is a subspace of $M$ with $w\in M'$.

Now for any  $v\in M'$ and $\beta\in \Z^m$, let $N=max\{|\beta_1|,\dots,|\beta_m|\}$. Then any $\alpha\in\Z^m$ with $\alpha_i>N(v)+N,\ i=1,\dots,m$, we have
$$ W_\alpha W_\beta v\subseteq W_\beta\g_\alpha v+[ W_\alpha, W_\beta]v\subseteq W_\beta W_\alpha v+ W_{\alpha+\beta}v=0.$$
So $ W_\beta w\in M'$.  $M'$ therefore is a nonzero submodule of the simple module $M$. Thus $M=M'$.
\end{proof}

\begin{lemma}\cite[Lemma 3.4]{MZ}\label{lem5.7}
Suppose $m>1$. Let $e_1,\dots,e_m,\lambda$ and $w$ be as in Lemma \ref{lem5.5}. Then $ W_{-\alpha}v\neq 0$ for any nonzero $v\in M$ and any $\alpha\in\N^m$.
\end{lemma}
\begin{proof}
Suppose $ W_{-\alpha} v=0$ for some nonzero homogeneous $v\in M$ and $\alpha\in\N^m$. By Lemma \ref{lem5.6}, there is a  $N\in \N$ such that $ W_{e_i+N\alpha}\cdot v=0,\ i=1,\dots,m$. Lemma \ref{lem5.1} implies that $ W$ is generated by $ W_{e_i+N\alpha},\ i=1,\dots,m$ and $ W_{-\alpha}$ as Lie superalgebra. Then $ W\cdot v=0$. This means that $M=\bC v$, which contradicts with the assumption that $M$ is not cuspidal. So the lemma holds.
\end{proof}

\begin{lemma}\cite[Lemma 3.5]{MZ}\label{lem5.8}
Suppose $m>1$. Let $e_1,\dots,e_m,\lambda$ and $w$ be as in Lemma \ref{lem5.5}. Then for any $\mu\in supp(M)$ and any $\alpha\in\N^n$ we have $\{k\in\Z|\mu+k\alpha\in supp(M)\}=\Z_{\leqslant m}$ for some $m\in\Z_+$.
\end{lemma}
\begin{proof} Let $S:=\{k\in\Z|\mu+k\alpha\in supp(V)\}$.
By Lemma \ref{lem5.7}, $S=\Z_{\leqslant m}$ for some $m\in\Z_+$ or $S=\Z$.  Let $B={\rm span}\{t^{k\alpha}d_{\mu} |k\in\Z\}$ be the subalgebra of $\Vir[\mu]$. Then $B\cong W_{1.0}$ and $V=M_{\mu+\Z\alpha}$ is a weight $B$ module with finite-dimensional weight spaces.
By Lemma \ref{lem5.6}, for any $v\in V$, there is a $k_0\in \N$ such that $t^{k\alpha}d\cdot v=0$ for all $k\geqslant k_0$. From Lemma \ref{lem5.2}, the support set of $B$ module $M$ is upper bounded. Hence $S\ne \Z$, and the lemma follows.\end{proof}

\begin{lemma}\cite[Lemma 3.6]{MZ}\label{lem5.9}
Suppose $m>1$. Let $e_1,\dots,e_m,\lambda$ and $w$ be as in Lemma \ref{lem5.5} . Then, after an appropriate change of variables $t_1,t_2,\ldots,t_m$, we have:
\begin{itemize}
\item[(1).] $ W_{e_1}\cdot w=0,i=1,\dots,m$.

\item[(2).]$\lambda+\alpha\notin {\rm Supp}(M)$ for any nonzero $\alpha\in\Z_+^m$.

\item[(3).]$\lambda-\alpha\in {\rm Supp}(M)$ for any $\alpha\in\Z_+^m$.

\item[(4).]For any $\alpha,\beta\in\Z^m$ such that $\alpha_i\leqslant\beta_i,\ i=1,\dots,m$, we have $\lambda+\alpha\notin {\rm supp}(M)$ implies that $\lambda+\beta\notin {\rm supp}(M)$.
\end{itemize}
\end{lemma}
\begin{proof}
By Lemma \ref{lem5.8}, there is an integer $p\geqslant 2$ such that $\{k\in\Z|\lambda+k(1,\dots,1)\in {\rm supp}(M)\}=\Z_{\leqslant p-2}$. Let $e'_1=(p+1)e_1+pe_2+\dots+pe_m,e'_2=e_1+e_2+e'_1,e'_i=e'_1+e_i,i=3,\dots,m$.
Then $e'_1,\dots,e'_m$ is another $\Z$-basis of $\Z^m$. Replace $e_i$ with $e'_i$ for all $i\in\{1,\dots,m\}$. Then (1) is clear, and (3) follows from Lemma \ref{lem5.8}.  (2) and (3) are proved in a similar way to the proof of Lemma \ref{lem5.7} by noting that $\lambda+(p-1)(1,1,\ldots,1)\not\in {\rm supp}(M)$. \end{proof}

\begin{lemma}\cite[Lemma 3.7]{MZ}\label{lem5.10}
Suppose $m>1$. Let $e_1,\dots,e_m,\lambda$ and $w$ be as in the Lemma \ref{lem5.9}. Then $M\cong L(G,\beta,X)$ for some subgroup $G$ of $\Z^m$ and $0\ne \beta\in \Z^m$ with $\Z^m=G\oplus \Z\beta$ and a simple cuspidal $ W_G$ module $X$.
\end{lemma}
\begin{proof}
$M$ is a weight $Vir[\mu]$ module with finite-dimensional weight spaces, then $M$ has a simple $Vir[\mu]$-subquotient $V$ with $V_\lambda\neq 0$. Since ${\rm supp} (V)\cap(\lambda+\Z_+^m)=\lambda$, $V$ is isomorphic to $L_{Vir[\mu]}(G,\beta,X)$ for some subgroup $G$ of $\Z^m$, nonzero $\beta\in\Z^m$ with $\Z^m=G\oplus\Z\beta$ and some simple cuspidal $Vir_G[\mu]$ module $X$. It follows that \begin{equation}\label{5.1}(\lambda-\Z_+\beta+G)\setminus\{0\}\subseteq {\rm supp}(M).\end{equation}

There exists $\alpha\in\N^m$ such that $G=\{\gamma\in\Z^m|(\gamma,\alpha)=0\}$. In fact, such $\alpha$ exists in $\Z^m$. If $\alpha_i=0$ for some $i\in\{1,\dots,m\}$, then $e_i\in G$. This contradicts with the fact that $(\lambda+G)\setminus\{0\}\subseteq {\rm supp}(M)$ and $\lambda+\gamma\notin {\rm supp}(M)$ for any nonzero $\gamma\in\Z_+^m$. So $\alpha_i\neq 0$ for all $i$. If  $\alpha_i\alpha_j<0$ for some $i,j\in\{1,\dots,m\}$, then $\alpha_je_i-\alpha_ie_j\in G$ will also lead to a contradiction. Hence we may assume that $\alpha\in\N^m$.

Case 1. $\{\lambda+k\beta+G\}\cap {\rm supp}(M)=\varnothing$ for some $k\in\N$.

Choose $k$ as minimal as possible. Let $X'=M_{\lambda+(k-1)\beta+G}$. Then $V\cong L(G,\beta,X')$ with $X'$, from Lemma \ref{lem5.4}, being a simple cuspidal $ W_G$ module.

Case 2. $\{\lambda+k\beta+G\}\cap {\rm supp}(M)\neq\varnothing$ for all $k\in\N$.

Let $k_0\in\Z$ such that $\alpha\in k_0\beta+G$. From Lemma \ref{lem5.9}, we have $\lambda+k\alpha\notin {\rm supp}(M)$ for any $k\in\N$.  This together with (\ref{5.1}) gives  $k_0\in \N$.

Since $\alpha\in (k_0\beta+G)\cap\N^m$ with $k_0\in\N$, we can choose sufficiently large $k$ such that

(I)$|\{\lambda+k\beta+G\}\cap\{\lambda+\Z_+^m\}|>1$

(II)$|\{\lambda+(k-1)\beta+G\}\cap\{\lambda+\Z_+^m\}|>1$.

Then $M_{\lambda+k\beta+G}$ is a weight $\Vir_G[\mu]$ module with finite-dimensional weight spaces.  From the assumption in (I), and the support set of a cuspidal module, we have $M_{\lambda+k\beta+G}$ is not cuspidal. From a same arguments as in the proof of  Lemma \ref{lem5.5}, that there is a $\Z$-basis $\beta_2,\dots,\beta_m$ of $G,\mu \in\lambda+k\beta+G$ and a nonzero homogeneous vector $v\in M_\mu$ such that $ W_{\beta_i}v=0,i=2,\dots,m$.

Let $\nu\in\{\lambda+(k-1)\beta+G\}\cap\{\lambda+\Z_+^m\}$ with $\nu\neq\lambda$, which exists by (II). Then $\nu\notin {\rm supp}(M)$ by Lemma \ref{lem5.9} (2). Let $\beta_1=\nu-\mu$, then $ W_{\beta_1}v=0$. Clearly, $\beta_1,\dots,\beta_m$ is a $\Z$-basis of $\Z^m$. By Lemma \ref{lem5.6}, $\mu+r(\beta_1+\dots+\beta_m)\notin {\rm supp}(M)$ for sufficiently large $r$. On the other hand, for $r\in \N$, we have
$$(\alpha, r(\beta_1+\cdots+\beta_m))=r(\alpha,\beta_1)=-r(\alpha,\beta)<0.$$

Hence, $\mu+r(\beta_1+\dots+\beta_m)\in(\lambda-\Z_+\beta+G)\setminus\{0\}$ for sufficiently large $r$. This contradicts with (\ref{5.1}). Hence Case 2 cannot happen. The claim of the lemma follows.
\end{proof}

\begin{theorem}\label{main} Let $(m,n)\in \N\times \Z_+$. Any simple weight $W=W_{m,n}$ module with finite-dimensional weight spaces is isomorphic to one of the following modules:

{\rm (1).} a  simple quotient of a tensor module $\Gamma(\lambda, V)$, where $\lambda\in \bC^m$ and $V$ is a finite-dimensional simple $\gl(m,n)$ module;

{\rm (2).}  a  module $L(G,e_m, X)^B$ of highest weight type, where $G=\Z e_1+\ldots+\Z e_{m-1}$, $B\in {\rm GL}_m(\Z)$, $X$ is  a simple quotient of the $W_{G}$ module $\Gamma(\lambda, V)$ for some  $\lambda\in \bC^m$ and  a finite-dimensional simple $\gl(m-1,n)$ module $V$.
\end{theorem}
\begin{proof} Let $M$ be any simple weight $W=W_{m,n}$ module with finite-dimensional weight spaces. If $M$ is cuspdial, we have (1)  from Theorem \ref{the3.11}.  Now suppose that $M$ is not cuspidal, then from Lemma \ref{lem5.3} and Lemma \ref{lem5.10}, we have $M\cong L(G,e_m, X)^B$, where $G=\Z e_1+\cdots+\Z e_{m-1}$, $B\in {\rm GL}_m(\Z)$, and $X$ is a simple cuspidal $W_G$ module.  From Theorem \ref{the4.4}, such $X$  is  a simple quotient of the $W_G$ module $\Gamma(\lambda, V)$ for some  $\lambda\in \bC^m$ and  a finite-dimensional simple $\gl(m-1,n)$ module $V$. We therefore have (2) in this case.  \end{proof}

{\bf Ackowledgement.} {This work is partially supported by NSF of China (Grant 11471233, 11771122, 11971440)}.


\vspace{4mm}

 \noindent R.L\"u: Department of Mathematics, Soochow University, Suzhou, P. R. China.  Email: rlu@suda.edu.cn, corresponding author.

\vspace{0.2cm}\noindent Y. Xue.: Department of Mathematics, Soochow University, Suzhou, P. R. China.  Email: yhxue00@stu.suda.edu.cn

\end{document}